\documentclass[10pt]{amsart}

\usepackage{amssymb}
\usepackage{bm}
\usepackage{graphicx}
\usepackage{psfrag}
\usepackage{color}
\usepackage{hyperref}
\hypersetup{colorlinks=true, linkcolor=blue, citecolor=magenta, urlcolor=wine}
\usepackage{url}
\usepackage{algpseudocode}
\usepackage{tikz-cd}
\usepackage{fancyhdr}
\usepackage{xy}
\input xy
\xyoption{all}
\usepackage{stmaryrd}
\usepackage{calrsfs}
\usepackage{mathabx}

\newtheorem*{maintheorem*}{Main Theorem}
\newtheorem{theorem}{Theorem}[section]

\newtheorem{question}[theorem]{Question}
\newtheorem{prop}[theorem]{Proposition}
\newtheorem{conj}[theorem]{Conjecture}
\newtheorem{lemma}[theorem]{Lemma}
\newtheorem{cor}[theorem]{Corollary}
\theoremstyle{definition}
\newtheorem{definition}[theorem]{Definition}
\newtheorem{remark}[theorem]{Remark}
\newtheorem{example}[theorem]{Example}
\numberwithin{equation}{section}

\newcommand{\nn}{\mathbb{N}}
\newcommand{\pp}{\mathbb{P}}
\newcommand{\qq}{\mathbb{Q}}
\newcommand{\rr}{\mathbb{R}}
\newcommand{\zz}{\mathbb{Z}}

\providecommand\ldb{\llbracket}
\providecommand\rdb{\rrbracket}
\newcommand\pval{\mathsf{v}_p}
\newcommand{\gp}{\text{gp}}

\newcommand{\uu}{\mathcal{U}}
\newcommand{\supp}{\text{supp}}


\begin{document}
	
	\mbox{}
	\title{Bi-atomic classes of positive semirings}
	
	\author{Nicholas R.~Baeth}
	\address{Department of Mathematics\\Franklin and Marshall College\\Lancaster, PA 17604}
	\email{nicholas.baeth@fandm.edu}
	
	\author{Scott T. Chapman}
	\address{Department of Mathematics and Statistics\\Sam Houston State University\\Huntsville, TX 77341}
	\email{scott.chapman@shsu.edu}
	
	\author{Felix Gotti}
	\address{Department of Mathematics\\MIT\\Cambridge, MA 02139}
	\email{fgotti@mit.edu}
	
	\date{\today}

	\begin{abstract}
		A subsemiring $S$ of $\rr$ is called a \emph{positive semiring} provided that $S$ consists of nonnegative numbers and $1 \in S$. Here we study factorizations in both the additive monoid $(S,+)$ and the multiplicative monoid $(S\setminus\{0\}, \cdot)$. In particular, we investigate when, for a positive semiring $S$, both $(S,+)$ and $(S\setminus\{0\}, \cdot)$ have the following properties: atomicity, the ACCP, the bounded factorization property (BFP), the finite factorization property (FFP), and the half-factorial property (HFP). It is well known that in the context of cancellative and commutative monoids, the chain of implications HFP $\Rightarrow$ BFP and FFP $\Rightarrow$ BFP $\Rightarrow$ ACCP $\Rightarrow$ atomicity holds. Here we construct classes of positive semirings wherein both the additive and multiplicative structures satisfy each of these properties, and we also give examples to show that, in general, none of the implications in the previous chain is reversible.
	\end{abstract}
	\bigskip

\maketitle

\bigskip
%%%%%%%%%%%%%%%%%%%
%%%%%%%%%%%%%%%%%%%
\section{Introduction} \label{sec:intro}

\smallskip

The atomic structure of \emph{positive monoids} (i.e., additive submonoids of the nonnegative cone of the real line) has been the subject of much recent investigation. The simplest positive monoids are numerical monoids (i.e., additive monoids consisting of nonnegative integers, up to isomorphism), and their additive structure has been investigated by many authors during the last three decades (see \cite{AG16,GR09} and references therein). In addition, the atomicity of Puiseux monoids (i.e., additive monoids consisting of nonnegative rational numbers) has been the subject of various papers during the last four years (see \cite{CGG21,CGG20} and references therein). Puiseux monoids have also been studied in connection to commutative rings~\cite{CG19} and factorizations of matrices~\cite{BG20}. Positive monoids can be thought of as natural higher-rank generalizations of Puiseux monoids, which are precisely the positive monoids of rank $1$. Submonoids of finite-rank free commutative monoids are also finite-rank positive monoids up to isomorphism, and their atomicity and arithmetic were studied in \cite{CKO00,CKO02} and, more recently, in~\cite{fG20,fG20a}. Positive monoids have also been considered in surprising contexts such as music theory; see the recent paper~\cite{mBA20}.

\smallskip

If a positive monoid contains $1$ and is closed under multiplication, then we call it a \emph{positive semiring}. Note that numerical monoids (other than $\nn_0$) are not positive semirings even though they are closed under multiplication. Nevertheless, factorization aspects of their multiplicative structure were recently investigated in~\cite{BE20}; we briefly consider this situation in Remark \ref{r:numerical-monoids}. Puiseux monoids generated by geometric sequences are positive semirings. They were systematically studied in \cite{CGG20} under the term ``cyclic rational semirings", where various factorization invariants of their additive structure were considered (the arithmetic of their multiplicative structure was briefly explored in~\cite{BG20}). The atomicity of certain generalizations of cyclic rational semirings was considered in~\cite{ABP21}. Furthermore, for a quadratic algebraic integer~$\tau$, both the elasticity and delta set of the multiplicative monoid of the positive semiring $\nn_0[\tau]$ were studied in~\cite[Section~3]{CCMS09}. On the other hand, the atomicity of the additive structure of the positive semiring $\nn_0[\alpha]$, where $\alpha$ is an algebraic integer, has been recently investigated in~\cite{CG20}. Finally, when $\alpha$ is transcendental, $\nn_0[\alpha]$ is isomorphic to the polynomial semiring $\nn_0[x]$, whose multiplicative structure was recently investigated in~\cite{CF19}.

\smallskip

A cancellative and commutative monoid is atomic if every non-invertible element factors into irreducibles, while it satisfies the ascending chain condition on principal ideals (or the ACCP) if every ascending chain of principal ideals eventually stabilizes, in which case, we also say that it is an ACCP monoid. One can easily show that monoids satisfying the ACCP are atomic. An atomic monoid is a bounded factorization monoid (or a BFM) if every element has only finitely many factorization lengths, while it is a finite factorization monoid (or an FFM) if every element has only finitely many factorizations (up to order and associates). Clearly, each FFM is a BFM, and it is not hard to argue that each BFM satisfies the ACCP. Finally, an atomic monoid is a half-factorial monoid (or an HFM) if every element has exactly one factorization length, while it is a unique factorization monoid (or a UFM) if every element has exactly one factorization (up to order and associates). Observe that each UFM is both an FFM and an HFM, and each HFM is a BFM. However, there are BFMs that are neither FFMs nor HFMs (see, for instance, \cite[Example~4.7]{AG20}). The implications mentioned in this paragraph are shown in Diagram~\eqref{diag:AAZ's chain for monoids}.

\smallskip

\begin{equation} \label{diag:AAZ's chain for monoids}
	\begin{tikzcd}%[cramped]
		\textbf{ UFM } \ \arrow[r, Rightarrow]  \arrow[d, Rightarrow] & \ \textbf{ HFM } \arrow[d, Rightarrow] \\
		\textbf{ FFM } \ \arrow[r, Rightarrow] & \ \textbf{ BFM } \arrow[r, Rightarrow]  & \textbf{ ACCP monoid}  \arrow[r, Rightarrow] & \textbf{ atomic monoid}
	\end{tikzcd}
\end{equation}

\smallskip

Diagram~\eqref{diag:AAZ's chain for monoids} was introduced in~\cite{AAZ90} by D. D. Anderson, D. F. Anderson, and M. Zafrullah in the setting of integral domains, where the authors illustrated that none of the involved implications is reversible. We call a positive semiring $S$ \emph{bi-atomic} (or \emph{bi-ACCP}, \emph{bi-BFS}, \emph{bi-FFS}, \emph{bi-HFS}, \emph{bi-UFS}) provided that both its additive monoid $(S,+)$ and its multiplicative monoid $(S \setminus \{0\}, \cdot)$ satisfy the corresponding property. The implications in Diagram~\eqref{diag:AAZ's chain for semirings} follow immediately from those in Diagram~\eqref{diag:AAZ's chain for monoids}.

\smallskip

\begin{equation} \label{diag:AAZ's chain for semirings}
	\begin{tikzcd}%[cramped]
		\textbf{ bi-UFS } \ \arrow[r, Rightarrow]  \arrow[d, Rightarrow] & \ \textbf{ bi-HFS } \arrow[d, Rightarrow] \\
		\textbf{ bi-FFS } \ \arrow[r, Rightarrow] & \ \textbf{ bi-BFS } \arrow[r, Rightarrow]  & \textbf{ bi-ACCP semiring}  \arrow[r, Rightarrow] & \textbf{ bi-atomic semiring}
	\end{tikzcd}
\end{equation}

\smallskip

In this paper, we use known results on the atomicity of positive monoids to gain a better insight of atomicity in certain classes of bi-atomic positive semirings. Our primary purpose here is twofold. On one hand, we identify classes of positive semirings that are bi-atomic, bi-ACCP, bi-BFSs, bi-FFSs, bi-HFSs, and bi-UFSs. On the other hand, we construct examples of positive semirings to illustrate that, as for the case of Diagram~\eqref{diag:AAZ's chain for monoids} in the class of (multiplicative monoids of) integral domains, none of the implications in Diagram~\eqref{diag:AAZ's chain for semirings} is reversible in the class of positive semirings. In order to accomplish our goals, we will consider simultaneously these properties for both the additive and multiplicative monoids of various classes of positive semirings. We emphasize, however, that we still do not know whether the topmost horizontal implication in Diagram~\eqref{diag:AAZ's chain for semirings} is reversible (see Conjecture~\ref{conj:bi-UFS} and Question~\ref{quest:bi-HFS/LFS}). In Section~\ref{sec:prelim}, we set the stage by recalling some basic definitions related to the content of this paper. We then consider atomicity (in Section~\ref{sec:atomicity}), the ascending chain condition on principal ideals (in Section~\ref{sec:ACCP}), the bounded factorization property (in Section~\ref{sec:BFP}), the finite factorization property (in Section~\ref{sec:FFP}), and finally the half-factorial property (in Section~\ref{sec:O/HFP}).

\bigskip
%%%%%%%%%%%%%%%%%%%%%%
%%%%%%%%%%%%%%%%%%%%%%
\section{General Notation and Definitions}
\label{sec:prelim}

In this section, we present the notation and definitions related to commutative monoids/semirings and factorization theory we will need for later sections. For a more comprehensive overview of these areas, the interested reader can consult~\cite{pG01} for commutative monoids, \cite{jsG99} for semirings,  and~\cite{GH06} for factorization theory.

\smallskip

We denote the empty set by $\emptyset$. The sets $\nn_0$, $\nn$, and $\pp$ will denote the set of nonnegative integers, positive integers, and prime numbers, respectively. In addition, the set of integers between $a, b \in \mathbb{Z}$ will be denoted by $\ldb a,b \rdb$, that is, $\ldb a,b \rdb = \{n \in \zz \mid a \le n \le b\}$ (observe that $\ldb a,b \rdb$ is empty when $a > b$). For $r \in \rr$ and $X \subseteq \rr$, we set $X_{\ge r} := \{x \in X \mid x \ge r\}$ and use the notations $X_{> r}, X_{\le r}$, and $X_{< r}$ in a similar manner. For $q \in \qq_{> 0}$, we refer to the unique $n,d \in \nn$ such that $q = n/d$ and $\gcd(n,d)=1$ as the \emph{numerator} and \emph{denominator} of $q$ and denote them by $\mathsf{n}(q)$ and $\mathsf{d}(q)$, respectively. Lastly, given $Q \subseteq \qq_{>0}$, we set $\mathsf{n}(Q) := \{\mathsf{n}(q) \mid q \in Q\}$ and $\mathsf{d}(Q) := \{\mathsf{d}(q) \mid q \in Q\}$.

\medskip
%%%%%%%%%%%%%%%%%%
\subsection{Commutative Monoids}

\smallskip

Throughout this manuscript, the term \emph{monoid} refers to a cancellative, commutative semigroup with identity. As we have the need to study both additive and multiplicative submonoids of $\mathbb R_{\geq 0}$, we will introduce the relevant factorization-theoretic terminology for a generic operation $\asterisk$. Let $(M,\asterisk)$ be a monoid with identity denoted by $\iota$. We set $M^\bullet := M \setminus \{\iota\}$, while we let $\uu(M)$ denote the group of invertible elements of $M$. We say that $M$ is \emph{trivial} when $M^\bullet = \emptyset$ and \emph{reduced} when $\uu(M) = \{\iota\}$. The \emph{Grothendieck group} $\gp(M)$ of $M$ is the unique abelian group $\gp(M)$ up to isomorphism satisfying that any abelian group containing a homomorphic image of $M$ will also contain a homomorphic image of $\gp(M)$. The \emph{rank} of $M$ is defined to be the rank of $\gp(M)$ as a $\zz$-module. The \emph{reduced monoid} of $M$ is the quotient of $M$ by $\uu(M)$, which is denoted by $M_{\text{red}}$. For $b, c \in M$, we say that $c$ \emph{divides} $b$ \emph{in} $M$ if there exists $d \in M$ such that $b = c \ast d$, in which case, we write $c \mid_M b$. A submonoid $M'$ of $M$ is a \emph{divisor-closed submonoid} if every element of $M$ dividing an element of $M'$ in $M$ belongs to $M'$.

\smallskip

If $A$ is a subset of $M$, then we let $\langle A \rangle$ be the submonoid of $M$ generated by the set $A$, that is, $\langle A \rangle = \left\{a_1 \ast \dots \ast a_n \mid n \in \nn_0 \text{ and } a_1, \dots, a_n \in A \right\}$, where the product of zero elements is the identity element. Clearly, $\langle A \rangle$ is the smallest submonoid of~$M$ containing $A$. If $M = \langle A \rangle$, we say that $A$ is a \emph{generating set} of $M$. The monoid $M$ is said to be finitely generated provided that $M = \langle A \rangle$ for some finite subset $A$ of $M$. An element $a \in M \setminus \uu(M)$ is an \emph{atom} if it cannot be written as $a = x \asterisk y$ for any two non-invertible elements $x,y \in M$. The set of atoms of $M$ will be denoted by $\mathcal{A}(M)$. Observe that if $M$ is reduced, then $\mathcal{A}(M)$ is contained in every generating set of $M$. Following \cite{CDM99}, we say that $M$ is \emph{antimatter} if $\mathcal{A}(M)$ is empty. We have a special interest in monoids that can be generated by their sets of atoms.

\smallskip

\begin{definition}
	A monoid $M$ is \emph{atomic} if each element in $M \setminus \uu(M)$ factors into atoms.
\end{definition}

\smallskip

Atomic monoids will play a crucial role in the upcoming discussion. One can readily verify that $M$ is atomic if and only if $M_{\text{red}}$ is atomic. In the context of monoids and integral domains, the ascending chain condition on principal ideals is a property that is often studied in connection to atomicity. An \emph{ideal} $I$ of~$M$ is a subset~$I$ of~$M$ satisfying $I \asterisk M := \{y \asterisk x \mid y \in I \text{ and } x \in M\} \subseteq I$ (or, equivalently, $I \asterisk M = I$). If~$I$ is an ideal of~$M$ such that $I = y \asterisk M := \{y \asterisk x \mid x \in M\}$ for some $y \in M$, then $I$ is \emph{principal}. In addition, $M$ satisfies the \emph{ascending chain condition on principal ideals} (or \emph{ACCP}) provided that each ascending chain of principal ideals of $M$ stabilizes. It is worth noting that any monoid satisfying the ACCP is atomic \cite[Proposition~1.1.4]{GH06}. However, the reverse implication does not hold in general; for instance, we will exhibit in Proposition~\ref{prop:infinitely many bi-atomic PCS without the ACCP} a class of additive submonoids of $(\qq,+)$ that are atomic but do not satisfy the ACCP. Atomic monoids and integral domains that do not satisfy the ACCP are not that easy to produce; the first such example was constructed by A.~Grams in~\cite{aG74}, and further examples were constructed by A.~Zaks in~\cite{aZ82}.

\smallskip

The monoid $M$ is \emph{free} on a subset $P$ of $M$ if every map $P \to M'$, where $M'$ is a monoid, uniquely extends to a monoid homomorphism $M \to M'$. For each set $P$, there is a unique free monoid on $P$ up to isomorphism. If $M$ is the free monoid on $P$, then every $x \in M$ can be written uniquely in the form $x = \Asterisk_{p \in P} \,  p^{\pval(x)}$, where $\pval(x) \in \nn_0$ and $\pval(x) > 0$ only for finitely many elements $p \in P$. Since the monoid $M$ is determined by $P$ up to isomorphism, we will sometimes denote $M$ by $\mathcal{F}(P)$. Note that the fundamental theorem of arithmetic can be simply stated as $\nn = \mathcal{F}(\pp)$, where $\mathbb N$ is considered multiplicatively.

\smallskip

Additive submonoids of $\rr_{\ge 0}$ play a central role throughout this paper. Following~\cite{fG19}, we call them \emph{positive monoids}. If a \emph{positive monoid} consists of rational numbers, then it is called a \emph{Puiseux monoid}. The atomic structure of Puiseux monoids has been systematically studied during the last few years (see, for instance, the recent paper~\cite{GGT21} and references therein). Puiseux monoids are, up to isomorphism, the positive monoids of rank $1$. Indeed, nontrivial Puiseux monoids can be characterized as the torsion-free rank-$1$ monoids that are not groups (see \cite[Theorem~3.12]{GGT21} and \cite[Section~24]{lF70}). Another important subclass of positive monoids is that consisting of \emph{numerical monoids}, that is, the additive cofinite submonoids of $\nn_0$ (see~\cite{GR09} for a treatment of numerical monoids and~\cite{AG16} for some of their applications). Numerical monoids account for all finitely generated Puiseux monoids, up to isomorphism.

\medskip
%%%%%%%%%%%%%%
\subsection{Factorizations}

Now assume that $M$ is atomic. The free (commutative) monoid on $\mathcal{A}(M_{\text{red}})$ is denoted by $\mathsf{Z}(M)$. Let $\pi \colon \mathsf{Z}(M) \to M_\text{red}$ be the unique monoid homomorphism fixing $\mathcal{A}(M_{\text{red}})$. If $z = a_1 \asterisk \cdots \asterisk a_\ell \in \mathsf{Z}(M)$, where $a_1, \dots, a_\ell \in \mathcal{A}(M_{\text{red}})$, then $\ell$ is called the \emph{length} of $z$ and is denoted by~$|z|$. For each $b \in M$, we set
\[
	\mathsf{Z}_M(b) = \pi^{-1} (b \asterisk \uu(M)) \quad \text{and} \quad \mathsf{L}_M(b) = \{ |z| \mid z \in \mathsf{Z}_M(b) \}.
\]
When we see no risk of ambiguity, we will simply write $\mathsf{Z}(b)$ and $\mathsf{L}(b)$ instead of the more cumbersome notation $\mathsf{Z}_M(b)$ and $\mathsf{L}_M(b)$, respectively. The sets $\mathsf{Z}(b)$ and $\mathsf{L}(b)$ play a crucial role in factorization theory (see~\cite{aG16}). The monoid $M$ is a \emph{finite factorization monoid} (or an \emph{FFM}) if $|\mathsf{Z}(b)| < \infty$ for every $b \in M$, while $M$ is a \emph{bounded factorization monoid} (or a \emph{BFM}) if $|\mathsf{L}(b)| < \infty$ for every $b \in M$. Each finitely generated monoid is an FFM by \cite[Proposition~2.7.8]{GH06}, and it is clear that each FFM is a BFM. In addition, $M$ is a \emph{unique factorization monoid} (or a \emph{UFM}) if $|\mathsf{Z}(b)| = 1$ for every $b \in M$, while $M$ is a \emph{half-factorial monoid} (or an \emph{HFM}) if $|\mathsf{L}(b)| = 1$ for every $b \in M$. A UFM is clearly both an FFM and an HFM, and an HFM is clearly a BFM. There are BFMs that are neither FFMs nor HFMs; see, for instance, \cite[Example~4.7]{AG20}. Finally, $M$ is a \emph{length-factorial monoid} (or an \emph{LFM}) provided that for all $b \in M$ and $z,z' \in \mathsf{Z}(b)$, the equality $|z| = |z'|$ implies that $z = z'$. The notion of length-factoriality was first considered in~\cite{CS11} under the term ``other-half-factoriality", and it has been recently investigated in~\cite{CCGS21}. For a recent survey on atomicity and factorizations in commutative monoids, see~\cite{GZ20}.

\medskip
%%%%%%%%%%%%
\subsection{Semirings}

Consider a triple $(S,+, \cdot)$, where $(S,+)$ is an additive monoid and $(S^\bullet, \cdot)$ is a multiplicative semigroup. If multiplication distributes over addition and the equalities $0 \cdot x = x \cdot 0 = 0$ hold for all $x \in S$, then $S := (S, +, \cdot)$ is a \emph{semiring}. If the semigroup $(S^\bullet, \cdot)$ is commutative, then $S$ is a \emph{commutative} semiring. Here we are interested in the atomic structure of subsemirings of $\rr_{\ge 0}$ with respect to the usual addition and multiplication of real numbers. Because we are primarily interested in atomicity, we will only consider subsemirings of $\rr_{\geq 0}$ containing $1$, the multiplicative identity.

\smallskip

\begin{definition}
	If a subsemiring of $\rr_{\geq 0}$ contains $1$, we call it a \emph{positive semiring}. If a positive semiring consists only of rational numbers (resp., algebraic numbers), we say it is a \emph{rational} (resp., an \emph{algebraic}) \emph{semiring}.
\end{definition}

\smallskip

Let $S$ be a positive semiring. Clearly, $S$ is commutative; indeed, $(S^\bullet, \cdot)$ is a monoid. %, which we will write as $(S^\bullet, \cdot)$.  
We call $(S,+)$ and $(S^\bullet, \cdot)$ the \emph{additive monoid} and the \emph{multiplicative monoid} of $S$, respectively. It is clear that $(S,+)$ is a reduced monoid. We say that a positive semiring $S$ is \emph{bi-reduced} provided that $(S^\bullet, \cdot)$ is reduced. We let $\mathcal{A}_+(S)$ and $\mathcal{A}_\times(S)$ denote  the sets of atoms of $(S,+)$ and $(S^\bullet, \cdot)$, respectively. In addition, we let $\uu(S)$ denote the set of elements of $S$ with multiplicative inverses.

\smallskip

\begin{definition}
	A positive semiring $S$ is \emph{bi-atomic} if both monoids $(S,+)$ and $(S^\bullet, \cdot)$ are atomic.
\end{definition}

\smallskip

Because the main focus of this paper is to study semirings whose additive and multiplicative monoids satisfy some of the well-known atomic properties defined in previous subsections, the following terminology will be useful in the upcoming sections. The positive semiring $S$ satisfies the \emph{bi-ACCP} provided that both monoids $(S,+)$ and $(S^\bullet, \cdot)$ satisfy the ACCP. In addition, $S$ is a \emph{bi-FFS} provided that both $(S,+)$ and $(S^\bullet, \cdot)$ are FFMs. In a similar manner, we use the terminologies \emph{bi-BFS}, \emph{bi-UFS}, \emph{bi-HFS}, and \emph{bi-LFS}.

\smallskip

Although our primary purpose is to investigate atomicity and further properties refining atomicity in positive semirings, we pause here to illustrate that one could consider atomic properties in a commutative semiring without multiplicative identity.

\smallskip

\begin{remark}\label{r:numerical-monoids}
	Observe that although every numerical monoid is a subsemiring of $\rr_{\ge 0}$, the only numerical monoid that contains $1$ and is thus a positive semiring is $\nn_0$. Let $S$ be a numerical monoid. Clearly, $(S,+)$ is an FFM. However, $(S,+)$ is an HFM if and only if $S = \nn_0$, in which case it is a UFM. In addition, $(S,+)$ is an LFM if and only if it can be generated by two elements (see~\cite[Example~2.13]{CS11} and \cite[Proposition~2.2]{fG20b}). It is similarly clear that $(S^\bullet, \cdot)$ is an FFM without identity, though should that make the reader squeamish, it does no harm to consider the monoid $(S^\bullet\cup\{1\}, \cdot)$ instead. From~\cite{BE20} we know that if $S \neq \nn_0$, then we can choose a $q \in \mathbb P \setminus S$ and an $n \in \nn_{\ge 2}$ so that whenever $p\in \mathbb P\cap S$, the elements $p, pq^{n-1}, q^n$, and $q^{2n-1}$ are multiplicative atoms of $S$. Since $p \cdot q^{2n-1} = (pq^{n-1}) \cdot q^n$, one immediately sees that $(S^\bullet, \cdot)$ is not an LFM. On the other hand, as $(q^n)^{2n-1} = (q^{2n-1})^n$, we observe that $(S^\bullet, \cdot)$ is not an HFM. As a result, every numerical monoid is bi-atomic, satisfies the bi-ACCP, and is both a bi-BFS and a bi-FFS. However, in the class of numerical monoids (treated as semirings), the equivalences bi-UFS $\Leftrightarrow$ bi-HFS  $\Leftrightarrow$ bi-LFS hold, and each of them is true precisely for the numerical monoid $\nn_0$.
\end{remark}

\medskip
%%%%%%%%%%%%%%
\subsection{Constructions}

\smallskip

Throughout, we will employ several constructions in order to both build examples of families of positive semirings satisfying various properties and to concoct examples of semirings showing that various implications do not hold. For convenience, we collect these basic constructions in this subsection.

\smallskip

\subsubsection*{Valuations of $\nn_0[x]$} By definition, every positive semiring contains $\nn_0$. Many of the examples that we consider are valuations of $\nn_0[x]$ at various $\alpha \in \mathbb R_{>0}$. Therefore it is convenient to introduce the following terminology.

\smallskip

\begin{definition}
	For each $\alpha \in \rr_{> 0}$, we call the semiring $\nn_0[\alpha] := \{f(\alpha) \mid f \in \nn_0[x]\}$ the \emph{cyclic positive semiring} generated by $\alpha$. When $\alpha$ is algebraic (resp., rational), we call $\nn_0[\alpha]$ a \emph{cyclic algebraic} (resp., \emph{rational}) \emph{semiring}.
\end{definition}

\smallskip

The additive structure of cyclic positive semirings was studied in detail in \cite{CG20}. We note that while the additive structure of $\nn_0[x]$ is simple (it is the free monoid on $\{x^n \mid n \in \nn_0\}$ and hence a UFM), the multiplicative structure, studied in \cite{CF19}, is far from simple. Nevertheless, the positive semirings $\nn_0[\alpha]$ are among the most tractable positive semirings; this is because the additive atomic structure of its members is relatively well understood, as we will corroborate throughout this paper.

\smallskip

\subsubsection*{Exponentiation} Often, it will be useful to construct, from a positive monoid $M$, a positive semiring whose multiplicative structure somehow reflects that of $M$. To do so, we invoke the following important result in transcendental number theory (see \cite[Chapter~1]{aB90}).

\smallskip

\begin{theorem}[Lindemann-Weierstrass Theorem]\label{t:L-W}
	If $\alpha_1, \dots, \alpha_n$ are distinct algebraic numbers, then the set 		$\{e^{\alpha_1}, \dots, e^{\alpha_n} \}$ is linearly independent over the algebraic numbers.
\end{theorem}

\smallskip

In particular, if $M$ is a positive monoid consisting of algebraic numbers, then the additive monoid $E(M) := \langle e^m \mid m \in M \rangle$\footnote{By the Lindemann-Weierstrass Theorem, $E(M)$ can be naturally identified with a subsemiring of the semigroup ring with coefficients in $\zz$ and exponents in $M$.} is free on the set $\{e^m \mid m \in M\}$. In addition, one can readily see that $E(M)$ is closed under the standard multiplication and, therefore, it is a positive semiring. The multiplicative structure of~$M$ is not, in general, as nice as its additive structure. It is thus sometimes convenient to consider the multiplicative submonoid $e(M) := \{e^m \mid m \in M\}$ of $E(M)^\bullet$, which is clearly isomorphic to the monoid~$M$. As we now see, $e(M)$ is divisor-closed in $E(M)$ and hence atoms and factorizations in the smaller monoid $e(M)$ persist to the larger monoid.

\smallskip

\begin{lemma}\label{l:e(M)}
	Let $M$ be a positive monoid. Then $e(M)$ is a divisor-closed submonoid of the multiplicative monoid of the semiring $E(M)$.
\end{lemma}

\begin{proof}
	Suppose that $e^m\in e(M)$ (necessarily with $m\in M$), and write $e^m = \big(\sum_{i=1}^m c_i e^{k_i} \big) \big(\sum_{j=1}^n c'_j e^{l_j} \big)$ in $E(M)$ assuming that $k_1 > \dots > k_m$ and $l_1 > \dots > l_n$ and taking the coefficients $c_1, \dots, c_m$ and $c'_1, \dots, c'_n$ to be positive integers. Because $(E(M),+)$ is the free monoid on $\{e^m \mid m \in M\}$ by Theorem~\ref{t:L-W}, we see that $e^m = (c_1e^{k_1}) (c'_1e^{l_1}) = c_1 c'_1 e^{k_1+l_1}$. As a result, $c_1 = c'_1 = 1$. Hence $\sum_{i=1}^m c_i e^{k_i} = e^{k_1}$ and $\sum_{j=1}^n c'_j e^{l_j} = e^{l_1}$ both belong to $e(M)$.
\end{proof}

\bigskip
%%%%%%%%%%
%%%%%%%%%%
\section{Atomicity}
\label{sec:atomicity}

\smallskip

The primary purposes of this section are to identify classes of bi-atomic positive semirings and to explore their atomic structure. These classes consist of real evaluations of the cyclic free semiring $\nn_0[x]$ of polynomials with nonnegative integer coefficients (various factorization aspects of the multiplicative structure of $\nn_0[x]$ were recently studied by F. Campanini and A. Facchini in~\cite{CF19}).

\smallskip

First, we present an easy result to illustrate how the multiplicative structure of a positive semiring interlaces with the atomicity of its additive structure.

\smallskip

\begin{prop} \label{prop:atoms of general positive semirings}
	For a positive semiring $S$, the following statements are equivalent.
	\begin{enumerate}
		\item[(a)] $\uu(S) \subseteq \mathcal{A}_+(S)$.
		\smallskip
		
		\item[(b)] $1 \in \mathcal{A}_+(S)$.
		\smallskip
		
		\item[(c)] $\mathcal{A}_+(S)$ is nonempty.
	\end{enumerate}
\end{prop}

\begin{proof}
	(a) $\Rightarrow$ (b) $\Rightarrow$ (c): Both implications follow immediately.
	\smallskip
	
	(c) $\Rightarrow$ (a): Suppose that $\mathcal{A}_+(S)$ is nonempty, and take $a \in \mathcal{A}_+(S)$. Write $1 = s+t$ for some $s,t \in S$. Since $a = sa + ta$ and $(S,+)$ is a reduced monoid, either $s=0$ or $t=0$. Hence $1 \in \mathcal{A}_+(S)$. Now take $u \in \uu(S)$, and suppose that $u = s_u + t_u$ for some $s_u, t_u \in S$. Then the equality $1 = u^{-1}s_u + u^{-1}t_u$ ensures that either $s_u = 0$ or $t_u = 0$, and so $u \in \mathcal{A}_+(S)$. Thus, $\uu(S) \subseteq \mathcal{A}_+(S)$.
\end{proof}

\smallskip

We now consider the atomicity of cyclic positive semirings. First, observe that not every such a semiring is bi-reduced: for instance, $2$ is a multiplicative unit of the positive semiring $\nn_0[1/2]$. Note, in addition, that $(\nn_0[1/2],+)$ is not atomic. Indeed, bi-reduced semirings $\nn_0[\alpha]$ can be characterized as those whose additive monoids are atomic.

\smallskip

\begin{prop} \label{prop:bi-reducedness characterization of PCS}
	For $\alpha \in \rr_{> 0}$, the positive semiring $\nn_0[\alpha]$ is bi-reduced if and only if $(\nn_0[\alpha],+)$ is atomic.
\end{prop}

\begin{proof}
	The statement of the proposition is clear when $\alpha = 1$. Thus, we assume that $\alpha \neq 1$. For the direct implication, assume that the semiring $\nn_0[\alpha]$ is bi-reduced. Suppose that $1 = \beta + \gamma$ in $(\nn_0[\alpha],+)$, and write $\beta = \sum_{i=0}^m b_i \alpha^i$ and $\gamma = \sum_{i=0}^n c_i \alpha^i$ for some $b_0, \dots, b_m, c_0, \dots, c_n \in \nn_0$. The equalities $b_0 = 0$ and $c_0 = 0$ cannot hold simultaneously as, in that case, $\alpha$ would be a multiplicative divisor of $1$ and, therefore, a unit in $(\nn_0[\alpha]^\bullet, \cdot)$. Hence either $b_0 \ge 1$ or $c_0 \ge 1$, and so $1 = \beta + \gamma$ implies that $\{\beta, \gamma\} = \{0,1\}$. Consequently, $1 \in \mathcal{A}_+(\nn_0[\alpha])$, and so it follows from \cite[Theorem~4.1]{CG20} that $(\nn_0[\alpha],+)$ is atomic.
	
	Conversely, suppose that $(\nn_0[\alpha],+)$ is atomic and so that $1 \in \mathcal{A}_+(\nn_0[\alpha])$. Let $u = \sum_{i=0}^m b_i \alpha^i$ be a multiplicative unit of the semiring $\nn_0[\alpha]$. Take $v = \sum_{i=0}^n c_i \alpha^i$ in $\nn_0[\alpha]^\bullet$ such that $uv = 1$. Since $1 \in \mathcal{A}_+(\nn_0[\alpha])$, it follows that $\sum_{i=0}^m b_i = \sum_{i=0}^n c_i = 1$. This implies that $u = \alpha^j$ and $v = \alpha^k$ for some $j,k \in \nn_0$. As $\alpha \neq 1$, the equality $1 = \alpha^{j+k}$ ensures that $j = k = 0$, and so $u=1$. Thus, $1$ is the only multiplicative unit of $\nn_0[\alpha]$, and so the positive semiring $\nn_0[\alpha]$ is bi-reduced.
\end{proof}

\smallskip

As the following example illustrates, Proposition~\ref{prop:bi-reducedness characterization of PCS} cannot be extended to the class of all positive semirings.

\smallskip

\begin{example}
	 Take $q \in \qq_{> 0}$ such that neither $q$ nor $q^{-1}$ is an integer. Since $\nn_0[x,y]$ is a semiring, its evaluation $\nn_0[q, q^{-1}]$ at $(x,y) = (q, q^{-1})$ is a positive semiring. It is not hard to verify that $1 \in \mathcal{A}_+(\nn_0[q,q^{-1}])$. Indeed, the equality $\mathcal{A}_+(\nn_0[q, q^{-1}]) = \{q^n \mid n \in \zz\}$ follows from the proof of \cite[Proposition~3.5]{fG18}, and so $(\nn_0[q,q^{-1}],+)$ is atomic. However, the semiring $\nn_0[q,q^{-1}]$ is not bi-reduced as, for instance, $q$ is a multiplicative unit.
\end{example}

\smallskip

As mentioned before, the atomic structure of $(\nn_0[\alpha], +)$ is very tractable. For instance, we will see in the next proposition that the set $\mathcal{A}_+(\nn_0[\alpha])$ is well structured and not hard to describe. For each $\alpha \in \rr_{> 0}$, set
\[
	n(\alpha) :=  \min \{n \in \nn \mid \alpha^n \in \langle \alpha^j \mid j \in \ldb 0,n-1 \rdb \rangle \} \in \nn \cup \{\infty\}
\]
if $(\nn_0[\alpha],+)$ is atomic, and set $n(\alpha) = 0$ otherwise.

\smallskip

\begin{prop}  \label{prop:additive atoms of CPS}
	Let $\alpha \in \rr_{> 0}$ be an algebraic number with minimal polynomial $m(x)$. Then the following statements hold.
	\begin{enumerate}
		\item $\mathcal{A}_+(\nn_0[\alpha]) = \{ \alpha^n \mid n \in \ldb 0, n(\alpha) - 1 \rdb\}$.
		\smallskip
		
		\item If $(\nn_0[\alpha],+)$ is atomic, then $\{ \alpha^n \mid n \in \ldb 0, \deg m(x) - 1 \rdb \} \subseteq \mathcal{A}_+(\nn_0[\alpha])$.
	\end{enumerate}
\end{prop}

\begin{proof}
	(1) If the additive monoid of the positive semiring $\nn_0[\alpha]$ is not atomic, then Proposition~\ref{prop:atoms of general positive semirings} ensures that $1 \notin \mathcal{A}_+(\nn_0[\alpha])$, from which the equality in~(1) trivially follows. The same equality holds from \cite[Theorem~4.1]{CG20} in the case when $(\nn_0[\alpha],+)$ is atomic.
	\smallskip

	(2) Note that when $(\nn_0[\alpha],+)$ is atomic, $1 \in \mathcal{A}_+(\nn_0[\alpha])$ by Proposition~\ref{prop:atoms of general positive semirings}. If $\alpha = 1$, then $\nn_0[\alpha] = \nn_0$, and the inclusion follows trivially. For $\alpha \neq 1$, we consider two cases:
	\smallskip
	
	\noindent {\it Case 1:} $\alpha < 1$. Fix $j \in \ldb 0, \deg m(x) - 1 \rdb$. Since $\alpha < 1$, we see that $\alpha^n$ does not divide $\alpha^j$ in $\nn_0[\alpha]$ for any $n \in \nn_{< j}$. Now suppose that $\alpha^j = \sum_{i=j}^n c_i \alpha^i$ for some $n \in \nn_{\ge j}$ and $c_j, \dots, c_n \in \nn_0$ with $c_n > 0$. If $n > j$, then $1 = \sum_{i=j}^n c_i \alpha^{i-j}$, in which case, $1 \notin \mathcal{A}_+(\nn_0[\alpha])$. So $n = j$ and $c_j = 1$. Thus, $\alpha^j \in \mathcal{A}_+(\nn_0[\alpha])$.
	\smallskip
	
	\noindent {\it Case 2:} $\alpha > 1$. Suppose, by way of contradiction, that $\alpha^j \notin \mathcal{A}_+(\nn_0[\alpha])$ for some $j \in \ldb 0, \deg m(x) - 1 \rdb$. This implies that  $j \ge 1$. Observe that $\alpha^n$ does not divide $\alpha^j$ in $\nn_0[\alpha]$ for any $n \in \nn$ with $n > j$ because $\alpha > 1$. Therefore $\alpha^j = \sum_{i=0}^{j-1} c_i \alpha^i$ for $c_0, \dots, c_{j-1} \in \nn_0$, and so $\alpha$ is a root of the polynomial $f(x) = x^j - \sum_{i=0}^{j-1} c_i x^i \in \qq[x]$. Then $f(x)$ is a nonzero polynomial of degree strictly less than $\deg m(x)$ having $\alpha$ as a root. However, this contradicts the minimality of $m(x)$. As a consequence, $\{ \alpha^n \mid n \in \ldb 0, \deg m(x) - 1 \rdb \} \subseteq \mathcal{A}_+(\nn_0[\alpha])$.
\end{proof}

\smallskip

The following corollary, which was first proved in~\cite{CG20}, is an immediate consequence of  Proposition~\ref{prop:atoms of general positive semirings} and part~(1) of Proposition~\ref{prop:additive atoms of CPS}.

\smallskip

\begin{cor} \label{cor:atomic characterization of cyclic algebraic semirings}
	For $\alpha \in \rr_{> 0}$, the additive monoid of the semiring $\nn_0[\alpha]$ is atomic if and only if $\mathcal{A}_+(\nn_0[\alpha])$ is nonempty, in which case, $1 \in \mathcal{A}_+(\nn_0[\alpha])$.
\end{cor}

\smallskip

The additive structure of cyclic rational semirings $\nn_0[q]$ (for $q \in \qq_{>0}$), along with various factorization invariants, has been recently investigated in~\cite{CGG20}.

\smallskip

\begin{example} \label{ex:multiplicatively cyclic monoids atomicity}
	For $q \in \qq_{> 0}$, consider the monoid $\nn_0[q]$. Assume that $q \notin \nn$ as, otherwise, $\nn_0[q] = \nn_0$ has trivial atomic structure. If $\mathsf{n}(q) = 1$, then $1 = \mathsf{d}(q)q$ and so it follows from Proposition~\ref{prop:cyclic atomic characterization} that $\nn_0[q]$ is antimatter. Then assume that $\mathsf{n}(q) > 1$. One can readily verify that $1 = \sum_{i=1}^n c_i q^i$ for some $c_1, \dots, c_n \in \nn_0$ would imply that $\mathsf{n}(q) \mid \mathsf{d}(q)$, which is not possible. Hence $1 \in \mathcal{A}_+(\nn_0[q])$, and so $\nn_0[q]$ is atomic by Corollary~\ref{cor:atomic characterization of cyclic algebraic semirings}. Similarly, one can check that $q^n = \sum_{i=0}^{n-1} c_i q^i$ for some $c_0, \dots, c_{n-1} \in \nn_0$ would imply that $\mathsf{d}(q) \mid \mathsf{n}(q)$. Thus, it follows from Proposition~\ref{prop:additive atoms of CPS} that $\mathcal{A}_+(\nn_0[q]) = \{q^n \mid n \in \nn_0\}$.
\end{example}

\smallskip

We proceed to identify two classes of bi-atomic semirings.

\smallskip

\begin{prop} \label{prop:cyclic atomic characterization} \label{prop:cyclic atomicity sufficient conditions}
	For $\alpha \in \rr_{> 0}$, the following statements hold.
	\begin{enumerate}
		\item If $\alpha$ is transcendental, then the semiring $\nn_0[\alpha]$ is bi-atomic.
		\smallskip
		
		\item If $\alpha$ is algebraic and $\alpha \ge 1$, then the semiring $\nn_0[\alpha]$ is bi-atomic.
	\end{enumerate}
\end{prop}

\begin{proof}
	(1) When $\alpha$ is transcendental, the kernel of the ring homomorphism $\varphi \colon \zz[x] \to \rr$ consisting in evaluating at $\alpha$ is trivial. As $\varphi(\nn_0[x]) = \nn_0[\alpha]$, we obtain that $\nn_0[\alpha] \cong \nn_0[x]$ as semirings. The monoid $(\nn_0[x], +)$ is free and thus atomic. On the other hand, a simple degree consideration shows that $(\nn_0[x]^\bullet, \cdot)$ is atomic. Hence the semiring $\nn_0[\alpha]$ is bi-atomic.
	\smallskip
	
	(2) If $\alpha = 1$, then $\nn_0[\alpha] = \nn_0$ is clearly bi-atomic. Suppose now that $\alpha > 1$. Since for each $n \in \nn$, the interval $[0,n]$ contains only finitely many elements of $\nn_0[\alpha]$, the monoid $(\nn_0[\alpha],+)$ must be atomic (this argument will become more transparent in Theorem~\ref{thm:BF sufficient condition}). For a similar reason, $(\ln(\nn_0[\alpha]^\bullet),+)$ is atomic (as usual, $``\ln"$ denotes the logarithm base $e$). Since the monoid $(\nn_0[\alpha]^\bullet, \cdot)$ is clearly isomorphic to $ (\ln(\nn_0[\alpha]^\bullet),+)$, it must be atomic. Thus, $\nn_0[\alpha]$ is bi-atomic.
\end{proof}

\smallskip

With the notation as in Proposition~\ref{prop:cyclic atomic characterization}, it is also true that $\nn_0[\alpha]$ is bi-atomic even for some values $\alpha<1$. For instance, we shall see in Proposition~\ref{prop:infinitely many bi-atomic PCS without the ACCP} that the rational cyclic semiring $\nn_0[2/3]$ is bi-atomic. It is known that if $S$ is a bi-atomic rational semiring, then $|\mathcal{A}_+(S)| \in \{0,1,\infty\}$ (see \cite[Proposition~3.6]{BG20}). This is no longer the case for algebraic semirings, as the following proposition indicates.

\smallskip

\begin{prop}
	For every $n \in \nn \cup \{\infty\}$, there exists a bi-atomic algebraic semiring $S$ such that $|\mathcal{A}_+(S)| = n$.
\end{prop}

\begin{proof}
	For $n = \infty$, take $q \in \qq_{> 1} \setminus \nn$ and consider the rational semiring $\nn_0[q]$. It follows from Proposition~\ref{prop:cyclic atomic characterization} that $\nn_0[q]$ is a bi-atomic semiring. In addition, we have seen in Example~\ref{ex:multiplicatively cyclic monoids atomicity} that $\mathcal{A}_+(\nn_0[q]) = \{q^n \mid n \in \nn_0\}$. Thus, $|\mathcal{A}_+(\nn_0[q])| = \infty$.
	
	Now suppose that $n \in \nn$. Consider the polynomial $m(x) = x^n - p$, where $n \in \nn$ and $p \in \pp$. Because $m(1) < 0$, the polynomial $m(x)$ has a root $\alpha \in \rr_{>1}$. It follows from part~(2) of Proposition~\ref{prop:cyclic atomic characterization} that the positive semiring $\nn_0[\alpha]$ is bi-atomic. We claim that $|\mathcal{A}_+(\nn_0[\alpha])| = n$. By Eisenstein's Criterion, the polynomial $m(x)$ is irreducible in $\qq[X]$, and so it is the minimal polynomial of $\alpha$. Therefore $\{\alpha^j \mid j \in \ldb 0, n - 1 \rdb \}$ is contained in $\mathcal{A}_+(\nn_0[\alpha])$ by Proposition~\ref{prop:additive atoms of CPS}. On the other hand, $\alpha^n \notin \mathcal{A}_+(\nn_0[\alpha])$ because it is the sum of $p$ copies of $1$. Then it follows from part~(1) of Proposition~\ref{prop:additive atoms of CPS} that $\mathcal{A}_+(\nn_0[\alpha]) = \{\alpha^j \mid j \in \ldb 0, n - 1 \rdb \}$, whence $|\mathcal{A}_+(\nn_0[\alpha])| = n$.
\end{proof} 

\smallskip

Though we are primarily interested in bi-atomic positive semirings, we note that there are positive semirings whose additive monoids (resp., multiplicative monoids) are antimatter even though their multiplicative monoids (resp., additive monoids) are atomic. We conclude this section with two examples illustrating this fact.

\smallskip

\begin{example} \label{ex:a positive semiring with additve monoid antimatter and multiplicative monoid UFM}
	Consider the positive semiring $\nn_0[1/d]$, where $d \in \nn_{\ge 2}$. Clearly, $1 \notin \mathcal{A}_+(\nn_0[1/d])$ and so it follows from Corollary~\ref{cor:atomic characterization of cyclic algebraic semirings} that $(\nn_0[1/d],+)$ is antimatter. We now prove that $(\nn_0[1/d]^\bullet, \cdot)$ is atomic. Let $D$ be the set consisting of all prime divisors of $d$, and consider the multiplicative submonoid $M_D$ of $\nn$ generated by $D$. Since  $\nn_0[1/d]^\bullet = \{ n/d^m \mid m,n \in \nn_0 \}$, the group of units of $(\nn_0[1/d]^\bullet, \cdot)$ is $\gp(M_D)$. Let $M_P$ be the multiplicative submonoid of $\nn$ generated by $\pp \setminus D$, and define $\varphi \colon (\nn_0[1/d]^\bullet, \cdot) \to M_P$ by $n/d^m = n'$, where for $n \in \nn$ and $m \in \nn_0$, the element $n'$ is the maximal divisor of $n$ in $M_P$. It is clear that $\varphi$ is a surjective monoid homomorphism, and it is easy to check that $\varphi(x) = \varphi(y)$ if and only if $x/y \in \gp(M_D)$. It follows from the First Isomorphism Theorem (for monoids) that the reduced monoid of $(\nn_0[1/d]^\bullet, \cdot)$ is isomorphic to $M_P$. Since $M_P$ is a free (commutative) monoid, $(\nn_0[1/d]^\bullet, \cdot)$ must be a UFM and, therefore, an atomic monoid. In particular, the positive semiring $\nn_0[1/d]$ is not bi-atomic.
\end{example}

\smallskip

\begin{example} \label{ex:a positive semiring with additive monoid BFM and and multiplicative monoid antimatter}
	Consider the positive semiring $S = \{0\} \cup \qq_{\ge 1}$. One can easily check that $(S,+)$ is an atomic monoid and that $\mathcal{A}_+(S) = \qq \cap [1,2)$ (we shall discuss the atomic structure of a more general case in Theorem~\ref{thm:BF sufficient condition}). Let us verify that $\mathcal{A}_\times(S)$ is the empty set. First, note that $S$ is bi-reduced. For every $q \in \qq_{> 1}$, take $n \in \nn$ sufficiently large such that $q \cdot \frac{n}{n+1} > 1$, and so $q \cdot \frac{n}{n+1} \in S^\bullet$. This, along with the fact that $q = \big(q \frac{n}{n+1}\big) \big(\frac{n+1}{n}\big)$, ensures that $q \notin \mathcal{A}_\times(S)$. Hence $(S^\bullet,\cdot)$ is antimatter. In particular, the positive semiring $S$ is not bi-atomic.
\end{example}

\bigskip
%%%%%%%%%%%%%%%%%%%%%%%%%%%%%%%
%%%%%%%%%%%%%%%%%%%%%%%%%%%%%%%
\section{The Ascending Chain Condition on Principal Ideals}
\label{sec:ACCP}

\smallskip

In this section, we construct a class of positive semirings satisfying the bi-ACCP, and we identify a class of bi-atomic cyclic positive semirings that do not satisfy the bi-ACCP. We construct the former class in the next proposition, which is motivated by~\cite{BG20}. For such a construction, the use of Lindemann-Weierstrass Theorem (Theorem \ref{t:L-W}) is crucial.

\smallskip

\begin{prop} \label{prop:class of positive semirings with thebi-ACCP}
	Let $P$ be a nonempty subset of $\pp$, and let $M = \langle 1/p \mid p \in P \rangle$. Then the positive semiring $E(M)$ satisfies the bi-ACCP. %\subseteq \mathbb P$ of prime numbers, the (\textcolor{red}{TODO:} semiring generated by the) additive monoid $\langle e^{1/p} \mid p \in P \rangle$ is a positive semiring satisfying the bi-ACCP. 
\end{prop}

\begin{proof}
	Set $S = E(M)$. It can be deduced from \cite[Example~2.1]{AAZ90} that the monoid $M$ satisfies the ACCP. On the other hand, it follows from Theorem \ref{t:L-W} that $(S,+)$ is a free (commutative) monoid on the set $\{e^q \mid q \in M\}$. The fact that $S \cap (0,1)=\emptyset$ ensures that $S$ is a bi-reduced semiring. Since $(S,+)$ is a free (commutative) monoid, it must satisfy the ACCP. %and is, therefore, atomic.
	
	We proceed to argue that $(S^\bullet, \cdot)$ also satisfies the ACCP. To do so, let $(x_n S^\bullet)_{n \in \nn}$ be an increasing sequence of principal ideals of $(S^\bullet, \cdot)$. For each $n \in \nn$, take $y_{n+1} \in S^\bullet$ such that $x_{n+1}y_{n+1} = x_n$. Let $\ell$ be the limit of the decreasing sequence $(x_n)_{n \in \nn}$. From $\inf S^\bullet = 1$, we see that $\ell \ge 1$ and, as a result, $\lim y_n = 1$. We may therefore assume that $y_n < 2$ for every $n \in \nn_{\ge 2}$. Clearly, $S \cap (0,2)$ is a subset of $e(M) = \{e^q \mid q \in M\}$. For each $n \in \nn_{\ge 2}$, write $y_n = e^{q_n}$ for some $q_n \in M$. Take $c_1, \dots, c_k \in \nn$ and $r_1, \dots, r_k \in M$ with $r_1 < \dots < r_k$ such that $x_1 = \sum_{i=1}^k c_i e^{r_i}$. Since $(S,+)$ is a free (commutative) monoid on the set $\{e^{q} \mid q \in M\}$ and $(y_2 \cdots y_n) x_n = x_1$ for every $n \in \nn_{\ge 2}$, one finds that $x_n = \sum_{i=1}^k c_i e^{r_{n,i}}$ for $r_{n,1}, \dots, r_{n,k} \in M$ such that $r_j = r_{n,j} + \sum_{i=2}^n q_i$ for each $j \in \ldb 1, k \rdb$. As for each $j \in \ldb 1,k \rdb$, the equality $r_{n,j} = r_{n+1,j} + q_{n+1}$ holds for every $n \in \nn$, the sequence $(r_{n,j} + M)_{n \in \nn}$ is an ascending chain of principal ideals of~$M$. Since $M$ satisfies the ACCP, $(r_{n,j} + M)_{n \in \nn}$ must stabilize. Thus, $(x_n S^\bullet)_{n \in \nn_0}$ also stabilizes. Consequently, the monoid $(S^\bullet, \cdot)$ satisfies the ACCP, and so the positive semiring $S$ satisfies the bi-ACCP.
\end{proof}

\smallskip

Our next goal is to identify a class of bi-atomic cyclic rational semirings whose members do not satisfy the bi-ACCP. We first introduce some notation and a lemma.

\smallskip

Let $f(x) = \sum_{i=0}^n c_ix^{m_i}$ be a nonzero polynomial in $\qq[x]$. We say that $f(x)$ is \emph{canonically represented} by $\sum_{i=0}^n c_ix^{m_i}$ if $c_i \neq 0$ for every $i \in \ldb 0,n \rdb$ and $m_i \neq m_j$ for any distinct $i,j \in \ldb 0, n \rdb$; in this case, we call $\sum_{i=0}^n c_ix^{m_i}$ the \emph{canonical representation} of $f(x)$ (which is clearly unique up to commutation of its monomials). The \emph{support} $\supp \, f(x)$ of $f(x)$ is the set of exponents of the monomials appearing in its canonical representation. Let $\alpha$ be an algebraic number, and let $m(x)$ be the minimal polynomial of $\alpha$. Observe that there exists a unique $\ell \in \nn$ such that $\ell m(x) \in \zz[x]$ has content~$1$, while there are unique polynomials $m^+(x)$, $m^-(x) \in \nn_0[x]$ such that  $\ell m(x) = m^+(x) - m^-(x)$ and $\supp \, m^+(x) \bigcap \supp \, m^-(x) = \emptyset$. We call the pair $(m^+(x), m^-(x))$ the \emph{minimal pair} of $\alpha$.

\smallskip

The following lemma is a slight generalization of \cite[Theorem~4.7]{CG20}, and the proof we provide here follows the same idea.

\smallskip

\begin{lemma} \label{lem:ACCP necessary condition for additive monoids of PCS}
	Let $\alpha \in (0,1)$ be an algebraic number with minimal polynomial $m(x)$. If $(\nn_0[\alpha],+)$ satisfies the ACCP, then $m^+(x) \notin \nn_0[x] + m^-(x) \nn_0[x]^\bullet$.
\end{lemma}

\begin{proof}
	Suppose that the monoid $(\nn_0[\alpha],+)$ satisfies the ACCP. Assume, towards a contradiction, that $m^+(x) \in \nn_0[x] + m^-(x) \nn_0[x]^\bullet$, and then take polynomials $a(x) \in \nn_0[x]$ and $b(x) \in \nn_0[x]^\bullet$ satisfying $m^+(x) = a(x) + b(x)m^-(x)$. Let $s$ be a nonnegative integer in the support of $b(x)$, and observe that $c(x) := a(x) + (b(x) - x^s)m^-(x)$ belongs to $\nn_0[x]$. For each $n \in \nn$, let $I_n$ denote the principal ideal $m^-(\alpha) \alpha^{sn} + \nn_0[\alpha]$ of $(\nn_0[\alpha],+)$. Since the equality $m^-(\alpha) \alpha^{sn} = c(\alpha)\alpha^{sn} + m^-(\alpha)\alpha^{s(n+1)}$ holds for every $n \in \nn$, the sequence $(I_n)_{n \in \nn}$ is an ascending chain of principal ideals of $(\nn_0[\alpha],+)$. In addition, $(I_n)_{n \in \nn}$ does not stabilize since the sequence $(\min I_n)_{n \in \nn} = (m^-(\alpha) \alpha^{sn})_{n \in \nn}$ is strictly decreasing. However, this contradicts that $(\nn_0[\alpha],+)$ satisfies the ACCP.
\end{proof}

\smallskip

We are now in a position to construct a bi-atomic semiring that does not satisfy the bi-ACCP.

\smallskip

\begin{prop} \label{prop:infinitely many bi-atomic PCS without the ACCP}
	For every $q \in \qq \cap (0,1)$ such that $\mathsf{n}(q) > 1$ and $\mathsf{d}(q) \in \pp$, the positive semiring $\nn_0[q]$ is bi-atomic but does not satisfy the bi-ACCP.
\end{prop}

\begin{proof}
	From Example~\ref{ex:multiplicatively cyclic monoids atomicity} we know that $(\nn_0[q],+)$ is an atomic monoid. We now show that $(\nn_0[q]^\bullet, \cdot)$ satisfies the ACCP. Suppose that $(r_n \nn_0[q]^\bullet)_{n \in \nn}$ is an ascending chain of principal ideals of $(\nn_0[q]^\bullet, \cdot)$. For every $n \in \nn$, the fact that $r_{n+1}$ divides $r_n$ in $(\nn_0[q]^\bullet, \cdot)$ ensures that $\mathsf{n}(r_{n+1}) \le \mathsf{n}(r_n)$ and $\mathsf{d}(r_{n+1}) \le \mathsf{d}(r_n)$. Consequently, the ascending chain of principal ideals $(r_n \nn_0[q]^\bullet)_{n \in \nn}$ must stabilize. As a result, $(\nn_0[q]^\bullet, \cdot)$ satisfies the ACCP and is therefore atomic. Thus, the rational semiring $\nn_0[q]$ is bi-atomic. Arguing that $\nn_0[q]$ does not satisfy the bi-ACCP amounts to verifying that the monoid $(\nn_0[q],+)$ does not satisfy the ACCP. This is an immediate consequence of Lemma~\ref{lem:ACCP necessary condition for additive monoids of PCS} as the minimal pair of~$q$ is $(\mathsf{d}(q)x, \mathsf{n}(q))$ and $\mathsf{d}(q) > \mathsf{n}(q)$.
\end{proof}

\smallskip

Although positive semirings satisfying the bi-ACCP are clearly bi-atomic, Proposition~\ref{prop:infinitely many bi-atomic PCS without the ACCP} shows that the reverse implication does not hold in general. We emphasize this observation with the following diagram.

\smallskip

\begin{equation} \label{diag:atomic chain unti bi-ACCP}
	\begin{tikzcd}%[cramped]
		\textbf{ bi-ACCP semiring}  \arrow[r, Rightarrow] \arrow[red, r, Leftarrow, "/"{anchor=center,sloped}, shift left=1.7ex]  & \textbf{ bi-atomic semiring}
	\end{tikzcd}
\end{equation}

\bigskip
%%%%%%%%%%%%%%%%%%%%%%%
%%%%%%%%%%%%%%%%%%%%%%%
\section{The Bounded Factorization Property}
\label{sec:BFP}

\smallskip

In this section, we identify a class of positive semirings that are bi-BFSs. In addition, we construct a class of bi-ACCP positive semirings that are not bi-BFSs. It was first proved in \cite[Proposition~4.5]{fG19} that a positive monoid $M$ is a BFM provided that $0$ is not a limit point of $M^\bullet$. We proceed to establish a similar sufficient condition for a positive semiring to be a bi-BFS.

\smallskip

\begin{theorem} \label{thm:BF sufficient condition}
	Let $S$ be a positive semiring. Then $S$ is a bi-BFS provided that $1$ is not a limit point of $S^\bullet \setminus \{1\}$. In addition, for $r > 0$, the positive semiring $S_r$ generated by the ray $\rr_{\ge r}$ is a bi-BFS if and only if $r > 1$, in which case,
		\[
			\mathcal{A}_+(S_r) = \big( \{1\} \cup [r,r+1) \big) \setminus \{ \lceil r \rceil \} \quad \text{ and } \quad \mathcal{A}_\times(S_r) = \big( \pp_{< r^2} \cup [r,r^2) \big) \setminus \pp \cdot (S_r)_{> 1}.
		\]
\end{theorem}

\begin{proof}
	Suppose the set $S^\bullet \setminus \{1\}$ does not have $1$ as a limit point. For each $r \in [0,1) \cap S$ the sequence $(1 + r^n)_{n \in \nn}$ consists of elements of $S$ and converges to $1$. This ensures that $r = 0$ because $1$ is not a limit point of $S^\bullet \setminus \{1\}$. As a consequence, $(0,1) \cap S = \emptyset$. In particular, $0$ is not a limit point of $S^\bullet$ and, therefore, \cite[Proposition~4.5]{fG19} ensures that $(S,+)$ is a BFM. To show that $(S^\bullet, \cdot)$ is also a BFM, we use the fact that $(S^\bullet, \cdot)$ is isomorphic to $(\ln S^\bullet, +)$. Because $1$ is not a limit point of $S^\bullet \setminus \{1\}$, it follows that $0$ is not a limit point of $(\ln S^\bullet) \setminus \{0\}$. Then $(\ln S^\bullet, +)$ is a BFM by \cite[Proposition~4.5]{fG19}, and so $(S^\bullet, \cdot)$ is also a BFM. Thus, $S$ is a bi-BFS.
	
	For the direct implication of the second statement, it suffices to verify that whenever $r\leq 1$, either $(S_r,+)$ or $(S_r^\bullet, \cdot)$ is not atomic. When $r = 1$, the monoid $(S_r^\bullet, \cdot)$ is antimatter, as we have already seen in Example~\ref{ex:a positive semiring with additive monoid BFM and and multiplicative monoid antimatter}. Assume now that $r < 1$, and fix $s \in \rr_{> 0}$. Now take $n \in \nn$ such that $s/r^n > 1$. This implies that $s/r^n \in S_r$, and so $s = r^n(s/r^n) \in S_r$. Thus, $S_r$ is the positive semiring $\rr_{\ge 0}$, which is not a bi-BFS because $(\rr_{\ge 0},+)$ is not atomic. The reverse implication of the second statement is an immediate consequence of the first statement.
	
	Finally, we describe $\mathcal{A}_+(S_r)$ and $\mathcal{A}_\times(S_r)$ when $r > 1$. Notice that $S_r = \nn_0 \cup \rr_{\ge r}$. Since $(S_r,+)$ is a reduced monoid and $\rr_{\ge r+1} \subseteq 1 + S_r^\bullet$, one finds that $\mathcal{A}_+(S_r) \subseteq S^\bullet_r \cap \rr_{< r+1} =  \ldb 1, \lfloor r \rfloor \rdb \cup [r, r+1)$. Because $1 \in \mathcal{A}_+(S_r)$ and $m \notin \mathcal{A}_+(S_r)$ for any $m \in \ldb 2, \lceil r \rceil \rdb$,
	\[
		\mathcal{A}_+(S_r) = \big( \{1\} \cup [r,r+1) \big) \setminus \{ \lceil r \rceil \}.
	\]
	To determine $\mathcal{A}_\times(S_r)$, note first that $(S_r^\bullet, \cdot)$ is reduced and $(S_r)_{\ge r^2} \subseteq (S_r)_{> 1} \cdot (S_r)_{> 1}$. Consequently, $\mathcal{A}_\times(S_r) \subseteq (S_r^\bullet)_{< r^2} = \ldb 1, \lfloor r \rfloor \rdb \cup [r,r^2)$. Then we see that the only elements in $\ldb 1, \lfloor r \rfloor \rdb \cup [r,r^2)$ that are not in $\mathcal{A}_\times(S_r)$ are those that are properly divisible in $(S^\bullet_r, \cdot)$ by some prime number. Hence
	\[
		\mathcal{A}_\times(S_r) = \big( \pp_{< r^2} \cup [r,r^2) \big) \setminus \pp \cdot (S_r)_{> 1}.
	\]
\end{proof}

\smallskip

As the following example illustrates, the converse of the first statement of Theorem~\ref{thm:BF sufficient condition} does not hold.

\smallskip

\begin{example} \label{ex:an FF PM having 0 as a limit point}
	Consider the positive monoid $M = \big\langle \frac{\lfloor \sqrt p \rfloor }{p}  \mid p \in \pp \big\rangle$. Because $M$ is a submonoid of $\langle 1/p \mid p \in \pp \rangle$, it can be easily deduced from \cite[Example~2.1]{AAZ90} that $M$ satisfies the ACCP, and so that it is atomic. In addition, it is not hard to verify that $\mathcal{A}(M) = \big\{ \frac{\lfloor \sqrt{p} \rfloor}{p} \mid p \in \pp \big\}$. We proceed to show that $M$ is an FFM. Fix $q \in M^\bullet$, and suppose that $z$ is a factorization of $q$. Now take $p \in \pp$ such that $\sqrt{p} > \max \{q, \sqrt{\mathsf{d}(q)}\}$. Since $p$ does not divide $\mathsf{d}(q)$, the number of copies of the atom $\lfloor \sqrt{p} \rfloor /p$ that appear in $z$ must be a multiple of $p$, namely, $np$. Therefore
	\[
		n \lfloor \sqrt{p} \rfloor = np \frac{\lfloor \sqrt{p} \rfloor}{p} \le q.
	\]
	This, along with the inequality $\sqrt{p} > q$, implies that $n = 0$. Thus, $\lfloor \sqrt{p} \rfloor /p$ does not divide $q$ in~$M$. As a result, only finitely many atoms (or elements) in $M$ divide $q$, and so $\mathsf{Z}_M(q)$ is finite. Hence it follows from \cite[Theorem~2]{fHK92} that $M$ is an FFM and, in particular, a BFM.
	
	Now consider the positive semiring $S = E(M)$. It follows from~Theorem \ref{t:L-W} that $(S,+)$ is the free (commutative) monoid on the set $\{e^q \mid q \in M \}$ and thus a BFM. We proceed to show that $(S^\bullet, \cdot)$ is also a BFM. We can also deduce from Theorem~\ref{t:L-W} that $\nn$ is a divisor-closed submonoid of $(S^\bullet, \cdot)$. Therefore $\mathsf{L}_{S^\bullet}(m)$, which equals $\mathsf{L}_\nn(m)$, is a finite set for every $m \in \nn$. So take $x = \sum_{i=1}^n c_i e^{q_i} \in S^\bullet \setminus \nn$, where $c_1, \dots, c_n \in \nn$ and $q_1, \dots q_n \in M$ satisfy $q_1 > \cdots > q_n$. As $x \notin \nn$, we see that $q_1 > 0$. Write $x = m_1 \cdots m_k f_1 \cdots f_\ell$ for some $m_1, \dots, m_k \in \nn_{\ge 2}$ (allowing $k=0$, in which case, $m_1 \cdots m_k = 1$) and $f_1, \dots, f_\ell \in S \setminus \nn$. %We claim that $k + \ell \le \max \mathsf{L}_M(q_1) + \log_2 x$. 
	Because $x \ge m_1 \cdots m_k$ and $m_i \ge 2$ for every $i \in \ldb 1,k \rdb$, it follows that $k \le \log_2 x$. For each $i \in \ldb 1, \ell \rdb$, write $f_i = \sum_{j=1}^{n_i} c_{j,i}e^{q_{j,i}}$ for some $c_{1,i}, \dots, c_{n_i, i} \in \nn$ and $q_{1,i}, \dots,  q_{n_i,i} \in M$ with $q_{1,i} > \dots > q_{n_i,i}$. Observe that, for each $i \in \ldb 1, \ell \rdb$, the fact that $f_i \in S \setminus \nn$ guarantees that $q_{1,i} \in M^\bullet$. As $q_1 = \sum_{i=1}^\ell q_{1,i}$, we obtain that $\ell \le \max \mathsf{L}_M(q_1)$, and so $k+ \ell \le \log_2 x + \mathsf{L}_{M}(q_1)$. Since $M$ is a BFM and $q_1$ is uniquely determined by $x$, the set $\mathsf{L}_{S^\bullet}(x)$ is bounded. Thus, $(S^\bullet, \cdot)$ is also a BFM and, therefore, the positive semiring~$S$ is a bi-BFS. However, it is clear that $1$ is a limit point of $S^\bullet \setminus \{1\}$.
\end{example}

\smallskip

Now we construct a positive semiring that satisfies the bi-ACCP but is not a bi-BFS.

\smallskip

\begin{example}\label{ex:bi-ACCP positive semiring that is not bi-BFM}
	Let $M = \langle 1/p \mid p \in \pp \rangle$ and consider the positive semiring $S = E(M)$. By Proposition~\ref{prop:class of positive semirings with thebi-ACCP}, $S$ satisfies the bi-ACCP. Let us argue that $(S^\bullet, \cdot)$ is not a BFM. It is clear that $M$ is not a BFM (for instance, $p (1/p)$ is a length-$p$ factorization in $\mathsf{Z}_M(1)$ for every $p \in \pp$). On the other hand,~$M$ is isomorphic to the multiplicative monoid $e(M)$ and, therefore, $e(M)$ is not a BFM. By Lemma~\ref{l:e(M)}, the monoid $e(M)$ is a divisor-closed submonoid of $(S^\bullet, \cdot)$. Thus, the monoid $(S^\bullet, \cdot)$ is not a BFM. Hence $S$ is not a bi-BFS even though it satisfies the bi-ACCP.
\end{example}

\smallskip

As every bi-BFS satisfies the bi-ACCP, Example~\ref{ex:bi-ACCP positive semiring that is not bi-BFM} allows us to extend Diagram~\eqref{diag:atomic chain unti bi-ACCP} as follows.

\smallskip

\begin{equation} \label{diag:AAZ's atomic chain until bi-BFM}
	\begin{tikzcd}%[cramped]
		 \ \textbf{ bi-BFS } \arrow[r, Rightarrow]  \arrow[red, r, Leftarrow, "/"{anchor=center,sloped}, shift left=1.7ex] & \textbf{ bi-ACCP}  \arrow[r, Rightarrow] \arrow[red, r, Leftarrow, "/"{anchor=center,sloped}, shift left=1.7ex]  & \textbf{ bi-atomic}
	\end{tikzcd}
\end{equation}

\bigskip
%%%%%%%%%%%%%%%%%%%%%%
%%%%%%%%%%%%%%%%%%%%%%
\section{The Finite Factorization Property}
\label{sec:FFP}

\smallskip

Our next task is to introduce a class of positive semirings that are bi-FFSs, the class consisting of increasing positive semirings. Following \cite{fG19}, we say that a positive monoid $M$ is \emph{strongly increasing} if $\mathcal{A}(M)$ is the underlying set of a sequence that increases to infinity. Strongly increasing positive monoids have been considered in \cite{mBA20,BG20a,fG19}.

\smallskip

\begin{example}
	(1) One can easily see that the positive monoid $M_1 = \big\langle p + \frac{1}{p} \mid p \in \pp \big\rangle$ is atomic with $\mathcal{A}(M_1) = \{ p + \frac{1}{p} \mid p \in \pp \}$, which is clearly the underlying set of a divergent increasing sequence. Hence~$M_1$ is a strongly increasing positive monoid.
	\smallskip
	
	(2) On the other hand, the positive monoid $M_2 = \{0\} \cup \rr_{\ge 1}$ is atomic with $\mathcal{A}(M_2) = [1,2)$. Since $\mathcal{A}(M_2)$ is uncountable, it cannot be the underlying set of any sequence. As a consequence, $M_2$ is not strongly increasing.
	\smallskip
	
	(3) Now consider the positive monoid $M_3 = \big\langle \frac{n}{n+1} \mid n \in \nn \big\rangle$. One can readily verify that $M_3$ is atomic with $\mathcal{A}(M_3) = \{\frac{n}{n+1} \mid n \in \nn \}$. Even though $\mathcal{A}(M_3)$ is the underlying set of the increasing sequence $\big(\frac{n}{n+1}\big)_{n \in \nn}$, the monoid $M_3$ is not strongly increasing because the sequence $\big(\frac{n}{n+1}\big)_{n \in \nn}$ does not increase to infinity.
\end{example}

\smallskip

We call a positive semiring $S$ \emph{strongly increasing} provided that $(S,+)$ is a strongly increasing monoid. Every strongly increasing positive semiring is a bi-FFS, as we proceed to show. 

\smallskip

\begin{theorem} \label{thm:FF sufficient condition}
	Every strongly increasing positive semiring is a bi-FFS. In addition, if $M$ is a strongly increasing positive monoid consisting of algebraic numbers, then
	\[
		\mathcal{A}_+(E(M)) = \{e^r \mid r \in M\} \quad \text{ and } \quad \mathcal{A}_\times(E(M)) \supseteq \{e^a \mid a \in \mathcal{A}(M) \}.
	\]
\end{theorem}

\begin{proof}
	Let $S$ be a strongly increasing positive semiring. Then $(S,+)$ is a strongly increasing positive monoid, and it follows from \cite[Theorem~5.6]{fG19} that $(S,+)$ is an FFM. Since $S$ is a strongly increasing positive semiring, $0$ cannot be a limit point of $S^\bullet$. As a result, if $r \in S_{< 1}$, then the fact that the sequence $(r^n)_{n \in \nn}$ of~$S$ converges to $0$ enforces the equality $r = 0$. Hence $\inf S^\bullet = 1$, and so $(\ln S^\bullet, +)$ is a positive monoid. Because $S$ is strongly increasing, the set $S \cap [0,n]$ must be finite for every $n \in \nn$. Therefore there exists a strictly increasing sequence $(s_n)_{n \in \nn_0}$ with underlying set $S$. This implies that $(\ln s_n)_{n \in \nn}$ is an increasing sequence generating $(\ln S^\bullet, +)$, and so~\cite[Theorem~5.6]{fG19} ensures that $(\ln S^\bullet, +)$ is an FFM. Thus, $S$ is a bi-FFS.
	
	Now suppose that $M$ is a strongly increasing positive monoid consisting of algebraic numbers, and set $S = E(M)$. It follows from Theorem~\ref{t:L-W} that $(S,+)$ is the free (commutative) monoid with basis $\{e^r \mid r \in M\}$. As a result, $\mathcal{A}_+(S) = \{e^r \mid r \in M\}$. To argue the last inclusion, take $a \in \mathcal{A}(M)$ and write $e^a = fg$, where $f = \sum_{i=1}^m b_i e^{q_i}$ and $g = \sum_{j=1}^n c_j e^{r_j}$ for some coefficients $b_1, \dots, b_m$ and $c_1, \dots, c_n$ in $\nn$ and exponents $q_1, \dots, q_m$ and $r_1, \dots, r_n$ in $M$ with $q_1 > \dots > q_m \ge 0$ and $r_1 > \dots > r_n \ge 0$. Since $(S,+)$ is free on $\{e^r \mid r \in M\}$, after distributing the right-hand side of $e^a = fg$, one obtains that $q_1 + r_1 = q_m + r_n = a$, which implies that $m = n = 1$ and $b_1 = c_1 = 1$. Since $q_1 + r_1 = a \in \mathcal{A}(M)$, it follows that either $f = 1$ or $g = 1$, and so $e^a \in \mathcal{A}_\times(S)$.
\end{proof}

\smallskip

The converse of Theorem~\ref{thm:FF sufficient condition} does not hold, as we illustrate in the following example.

\smallskip

\begin{example} \label{ex:FF PM that is not increasing}
	Let $M$ be the positive monoid generated by the set
	\begin{equation} \label{eq:FF does not implies increasing}
		A = \bigg\{\frac{p_{2n}^2 + 1}{p_{2n}}, \frac{p_{2n+1} + 1}{p_{2n+1}} \ \bigg{|} \ n \in \nn \bigg\},
	\end{equation}
	where $(p_n)_{n \in \nn}$ is a strictly increasing sequence of primes. Now we consider the positive semiring $S = E(M)$. Because $(S,+)$ is the free (commutative) monoid on $\{e^r \mid r \in M\}$ by Theorem~\ref{t:L-W}, it follows that $\mathcal{A}_+(S) = \{e^r \mid r \in M\}$. Since $\mathcal{A}_+(S)$ is an unbounded subset of $\rr$ having $e$ as a limit point, it cannot be the underlying set of any increasing sequence. As each generating set of~$(S,+)$ must contain~$A$, the positive semiring $S$ is not strongly increasing.
	
	Since the denominators of any two different rationals in $A$ are distinct primes, it is not hard to verify that $A = \mathcal{A}(M)$, from which we can conclude that $M$ is atomic. Let us show, in fact, that $M$ is an FFM. To do this, we proceed as we did in Example~\ref{ex:an FF PM having 0 as a limit point}. Fix $q \in M^\bullet$ and then take $D_q$ to be the set of all prime numbers dividing $\mathsf{d}(q)$. Now choose $n_0 \in \nn$ such that $n_0 > \max \{q, \mathsf{d}(q)\}$. For each $a \in A$ such that $\mathsf{d}(a) > n_0$, the number $c$ of copies of the atom $a$ appearing in any factorization $z$ in $\mathsf{Z}_M(q)$ must be a multiple of $\mathsf{d}(a)$ because $\mathsf{d}(a) \notin D_q$. If $c$ were nonzero, $q \ge ca \ge \mathsf{d}(a)a > \mathsf{d}(a) > q$. Thus, $c = 0$. As a consequence, if an atom $a$ divides $q$ in $M$, then $\mathsf{d}(a) \le n_0$. Hence only finitely many elements of $\mathcal{A}(M)$ divide $q$ in $M$, and so~$M$ is an FFM by \cite[Theorem~2]{fHK92}.
	
	Finally, we show that $S$ is a bi-FFS. Since $(S,+)$ is a free (commutative) monoid by Theorem~\ref{t:L-W}, this amounts to verifying that the multiplicative monoid $(S^\bullet, \cdot)$ is an FFM. Fix $s = \sum_{i=1}^m b_i e^{q_i} \in S^\bullet \setminus \{1\}$ for some $b_1, \dots, b_m \in \nn$ and $q_1, \dots, q_m \in M$.  For every $i \in \ldb 1,m \rdb$, let $D_i$ denote the set of divisors of~$q_i$ in $M$. Because $M$ is an FFM, the set $D = D_1 \cup \dots \cup D_m$ is finite. Now set $b = \max\{b_1, \dots, b_m\}$. Let $r = \sum_{i=1}^n c_i e^{r_i} \in S^\bullet \setminus \{1\}$ be a divisor of $s$ in $(S^\bullet, \cdot)$, where $c_1, \dots, c_n \in \nn$ and $r_1, \dots, r_n \in M$. Since $(S,+)$ is free on $\{e^q \mid q \in M\}$, it is not hard to see that $r_i \in D$ and $c_i \le b$ for every $i \in \ldb 1,n \rdb$. Thus, $s$ has only finitely many divisors in $(S^\bullet, \cdot)$. As every element of $S^\bullet$ has finitely many divisors, it follows from \cite[Theorem~2]{fHK92} that $(S^\bullet, \cdot)$ is an FFM. Hence $S$ is a bi-FFS.
\end{example}

\smallskip

We conclude this section giving an example of a bi-BFS that is not a bi-FFS. %nor a bi-HFM.

\smallskip

\begin{example} \label{ex:bi-BFM not bi-FFM}
	Consider the positive semiring $S_2 = \nn_0 \cup \rr_{\ge 2}$. Note that $S_2$ is reduced because $1 = \inf S_2^\bullet$. It follows from Theorem~\ref{thm:BF sufficient condition} that $S_2$ is a bi-BFS satisfying that $\mathcal{A}_+(S_2) = \{1\} \cup (2,3)$ and $\mathcal{A}_\times(S_2) = [2,4)$. To verify that the additive monoid of $S_2$ is not an FFM, it suffices to take $r \in (4,5)$ and observe that the formal sum $(2 + 1/n) + (r - 2 - 1/n)$ is a length-$2$ factorization of $r$ in $(S_2,+)$ for every $n \in \nn$ with $n > \frac{1}{r-4}$. In a similar way, we can argue that the multiplicative monoid  $(S_2^\bullet, \cdot)$ is not an FFM. Hence $S_2$ is a bi-BFS that is not a bi-FFS. One can use a similar argument to show that, for each $r \ge 2$, the positive semiring $S_r = \nn_0 \cup \rr_{\ge r}$ is a bi-BFS that is not a bi-FFS.
\end{example}

\smallskip

It is clear that every bi-FFS is a bi-BFS, and we have just seen in Example~\ref{ex:bi-BFM not bi-FFM} that not every bi-BFS is a bi-FFS. Hence we can extend Diagram~\eqref{diag:AAZ's atomic chain until bi-BFM} as follows.

\smallskip

\begin{equation} \label{diag:AAZ's atomic chain until bi-FFM}
	\begin{tikzcd}%[cramped]
		\textbf{ bi-FFS } \arrow[r, Rightarrow]  \arrow[red, r, Leftarrow, "/"{anchor=center,sloped}, shift left=1.7ex] & \textbf{ bi-BFS } \arrow[r, Rightarrow]  \arrow[red, r, Leftarrow, "/"{anchor=center,sloped}, shift left=1.7ex] & \textbf{ bi-ACCP}  \arrow[r, Rightarrow] \arrow[red, r, Leftarrow, "/"{anchor=center,sloped}, shift left=1.7ex]  & \textbf{ bi-atomic}
	\end{tikzcd}
\end{equation}

\bigskip
%%%%%%%%%%%%%%%%%%%%%%%%%%%%%%%%
%%%%%%%%%%%%%%%%%%%%%%%%%%%%%%%%
\section{The Half-Factorial and the Length-Factorial Properties}
\label{sec:O/HFP}

\smallskip

In this final section, we consider the half-factorial and the length-factorial properties. We give two simple necessary conditions for a positive semiring to be a bi-HFS or a bi-LFS. These necessary conditions will allow us to provide several examples of positive semirings that are bi-BFSs but not bi-HFSs or bi-LFSs. In particular, they can be applied to some of the examples of positive semirings of the form $E(M)$ we have seen in previous sections. A rather striking example of a positive semiring whose multiplicative monoid is neither an HFM nor an LFM is $E(\nn_0)$; we will discuss this in detail in Example~\ref{ex:bi-FFS neither bi-HFS nor bi-LFS}. The section concludes with a final extended version of Diagram~\eqref{diag:AAZ's atomic chain until bi-FFM}.

\smallskip

\begin{prop}\label{prop:bi-HFS necessary conditions}
	If a positive semiring $S$ is a bi-HFS, then the following statements hold.
	\begin{enumerate}
		\item $S \cap \mathbb Q = \nn_0$.
		\smallskip
		
		\item If $S \cap (0,1) = \emptyset$ and $a \in \mathcal A_\times(S)$, then $S \cap \{a^q \mid q \in \qq\} = \{a^n \mid n \in \nn_0\}$.
	\end{enumerate}
\end{prop}

\begin{proof}
	(1) It is clear that $\nn_0 \subseteq S \cap \qq$. To argue the reverse inclusion, take $q \in S^\bullet \cap \mathbb Q$. Because $(S,+)$ is atomic, $q = \sum_{i=1}^k a_i$ for some $a_1, \dots, a_k \in \mathcal A_+(S)$. Then $\sum_{i=1}^k \mathsf d(q) a_i$ is a factorization of $\mathsf n(q)$  in $(S,+)$ of length $k \mathsf{d}(q)$. On the other hand, it follows from Proposition~\ref{prop:atoms of general positive semirings} that $1 \in \mathcal A_+(S)$, and so $\mathsf n(q) \cdot 1$ is a factorization of $\mathsf n(q)$ in $(S,+)$ of length $\mathsf{n}(q)$. Since $(S, +)$ is an HFM, the equality $\mathsf n(q) = k \mathsf d(q)$ holds, whence $q = k \in \nn$. As a result, $S \cap \qq \subseteq \nn_0$.
	\smallskip
	
	(2) Assume that $S \cap (0,1) = \emptyset$, and take $a \in \mathcal{A}_\times(S)$. Clearly, the monoids $(S^\bullet, \cdot)$ and $(\log_a S^\bullet, +)$ are isomorphic. As the set $S \cap (0,1)$ is empty, $(\log_a S^\bullet, +)$ is a positive monoid. In addition, we see that $1 \in \mathcal A (\log_a S^\bullet)$ because $a \in \mathcal A_\times(S)$. Mimicking the argument in the previous paragraph, one can verify that $(\log_a S^\bullet)\cap \mathbb Q = \nn_0$. As a result, $S \cap \{a^q \mid q \in \qq\} = \{a^n \mid n \in \nn_0\}$.
\end{proof}

\smallskip

We need $S$ to be a semiring in order to guarantee part~(1) of Proposition~\ref{prop:bi-HFS necessary conditions}, as the following example illustrates.

\smallskip

\begin{example}
	For $n \in \nn$, consider the positive monoid $M_n = \langle \pi, n, \frac{1}{2}(\pi + n) \rangle$. It is not hard to check that $\mathcal{A}(M_n) = \{ \pi, n, \frac{1}{2}(\pi + n) \}$. Since $\pi + n$ and $2 \frac{\pi + n}{2}$ are two distinct factorizations in $\mathsf{Z}(\pi + n)$, the monoid $M_n$ is not a UFM. In addition, suppose that
	\[
		c_1 \pi + c_2 n +c_3 \frac{\pi + n}2 = 0
	\]
	for some $c_1, c_2, c_3 \in \zz$. Then $c_1 + c_3/2 = 0$ and $c_2 + c_3/2 = 0$, from which the equality $c_1 + c_2 + c_3 = 0$ follows. Hence $M_n$ is an HFM. However, $M_n  \cap \qq = n \nn_0$.
\end{example}

\smallskip

On the other hand, there are positive monoids that are not semirings and still satisfy the condition in part~(1) of Proposition~\ref{prop:bi-HFS necessary conditions}.

\smallskip

\begin{example}
	Let $\omega$ be an irrational number with $1 < \omega < 2$ such that $\omega$ is not a quadratic integer. Consider the positive monoid $M = \langle q + (1-q)\omega \mid q \in \qq \cap [0,1] \rangle$. Since $q + (1-q)\omega \in [1,2)$ for each $q \in \qq \cap [0,1]$, it follows that $\mathcal{A}(M) = \{ q + (1-q)\omega \mid q \in \qq \cap [0,1] \}$ and, therefore, $M$ is atomic. It is clear that $\nn_0 \subseteq M \cap \qq$. To check the reverse inclusion, take $q \in M^\bullet \cap \qq$, and write $q = \sum_{i=1}^n q_i + (1-q_i)\omega$ for some $n \in \nn$ and $q_1, \dots, q_n \in \qq \cap [0,1]$. Since $1$ and $\omega$ are linearly independent over $\qq$, it follows that $q = \sum_{i=1}^n q_i$ and $0 = \sum_{i=1}^n (1-q_i)$. As a result, $q = \sum_{i=1}^n q_i = n \in \nn$. Hence $\nn_0 = M \cap \qq$.
	
	We proceed to argue that $M$ is an HFM that is not closed under multiplication. Fix $b \in M$, and consider two factorizations of lengths $k$ and $\ell$ in $\mathsf{Z}(b)$, that is,
	\[
		\sum_{i=1}^k q_i + (1-q_i)\omega = b = \sum_{j=1}^\ell q'_j + (1 - q'_j)\omega
	\]
	for some rational numbers $q_1, \dots, q_k$ and $q_1', \dots, q'_\ell$ in the interval $[0,1]$. The fact that $1$ and $\omega$ are linearly independent over $\qq$ immediately implies that both equalities $\sum_{i=1}^k q_i = \sum_{j=1}^\ell q'_j$ and $\sum_{i=1}^k (1 - q_i) = \sum_{j=1}^\ell (1 - q'_j)$ hold. After adding both equalities, one finds that $k = \ell$. Hence $M$ is an HFM. If $M$ were closed under multiplication, then $\omega^2 = a + b\omega$ for some $a,b \in \qq_{\ge 0}$, which is not possible because $\omega$ is an irrational number that is not a quadratic integer. As a final remark, observe that for all $q_1, q_2 \in \qq \cap [0,1]$ with $q_1 + q_2 = 1$, the equality $1 + \omega = \big( q_1 + (1- q_1)\omega \big) + \big( q_2 + (1 - q_2)\omega \big)$ holds, whence $M$ is not even an FFM.
\end{example}

\smallskip
\vspace{5pt}

The following proposition extends \cite[Proposition 4.25]{CGG21} to give an analog of Proposition~\ref{prop:bi-HFS necessary conditions} for positive semirings that are bi-LFS.

\smallskip

\vspace{5pt}

\begin{prop}\label{prop:bi-LFS necessary conditions}
	If a positive semiring $S$ is a bi-LFS, then the following statements hold.
	\begin{enumerate}
		\item $|\mathcal A_+(S) \cap \mathbb Q| \le 2$.
		\smallskip
		
		\item If $S\cap (0,1)=\emptyset$ and $a\in\mathcal A_\times(S)$, then $|\mathcal A_\times(S) \cap \{a^q \mid q \in \qq\}| \leq 2$.
	\end{enumerate}
\end{prop}

\begin{proof}
	(1) Since $(S,+)$ is an LFM, \cite[Theorem~3.1]{CCGS21} guarantees the existence of $a \in \mathcal{A}_+(S)$ such that the set $\mathcal{A}_+(S) \setminus \{a\}$ is integrally independent in the Grothendieck group of $(S,+)$. This, along with the fact that the group $\qq$ has rank $1$, immediately implies that at most two elements of $\mathcal A_+(S)$ can be rational numbers, that is, $|\mathcal A_+(S) \cap \mathbb Q| \le 2$. This part could have also been proved by mimicking the proof of \cite[Proposition~2.2]{fG20b}.
	\smallskip
	
	(2) Suppose that $S \cap (0,1) = \emptyset$, and take $a \in \mathcal{A}_\times(S)$. Take $a^{q_1}, a^{q_2}, a^{q_3} \in \mathcal{A}_\times(S)$ for some $q_1, q_2, q_3 \in \qq$. Since $(S^\bullet, \cdot) \cong (\log_a S^\bullet, +)$, it follows that $q_1, q_2, q_3 \in \mathcal{A}(\log_a S^\bullet)$. As $(\log_a S^\bullet, +)$ is an LFM, $|\{q_1, q_2, q_3\}| \le 2$ by the previous part (observe that we argued the previous part without appealing to the multiplicative structure of~$S$). Then we conclude that $|\mathcal A_\times(S) \cap \{a^q \mid q \in \qq\}| \leq 2$.
\end{proof}

\smallskip

Since we did not use the multiplicative structure of~$S$ to establish part~(1) of Proposition~\ref{prop:bi-LFS necessary conditions}, the following statement holds: $|\mathcal{A}(M) \cap \qq| \le 2$ for every length-factorial positive monoid~$M$. The same condition in part~(1) of Proposition~\ref{prop:bi-LFS necessary conditions} does not guarantee that a positive semiring is a bi-LFS. The following example sheds some light upon this observation.

\smallskip

\begin{example}
	The set $S := \nn_0 \cup \big((2,4) \setminus \qq \big) \cup \rr_{> 4}$ is closed under both addition and multiplication, and so $S$ is a positive semiring. In addition, Theorem~\ref{thm:BF sufficient condition} guarantees that $S$ is a bi-BFS and, in particular, a bi-atomic positive semiring. Now fix $r \in \qq_{> 4}$. Taking $\epsilon$ to be an irrational number so that $0 < \epsilon < r-4$, we can write $r = (2 + \epsilon) + (r-(2 + \epsilon) )$. Since both $2+ \epsilon$ and $r - (2 + \epsilon)$ are irrational numbers greater than $2$, it follows that $r \notin \mathcal{A}_+(S)$. As a result, the only rational additive atom of $S$ is~$1$, which implies that $|\mathcal{A}_+(S) \cap \qq| \le 2$. However, $(S,+)$ is not an LFM, as we proceed to verify. If $\alpha$ is an irrational number with $0 < \alpha < 1/2$, then $5/2 \pm \alpha \in \mathcal{A}_+(S)$ and so the equality $5 = (5/2 - \alpha) + (5/2 + \alpha)$ yields a length-$2$ factorization of $5$ in $(S,+)$. As distinct choices of $\alpha$ yield distinct length-$2$ factorizations of $5$, the monoid $(S,+)$ is not an LFM. Thus, $S$ is not a bi-LFS. As a side note, observe that $5$ also has a length-$5$ factorization in $(S,+)$, and so $S$ is not a bi-HFS. 
\end{example}

\smallskip

Propositions~\ref{prop:bi-HFS necessary conditions} and~\ref{prop:bi-LFS necessary conditions} immediately give rise to a wealth of examples of positive semirings that are neither bi-HFSs nor bi-LFSs. Consider, as evidence, the positive semirings in Examples~\ref{ex:an FF PM having 0 as a limit point}, \ref{ex:FF PM that is not increasing}, and~\ref{ex:bi-BFM not bi-FFM}, and the positive semirings $S_r$ in Theorem~\ref{thm:BF sufficient condition}. We note that the exponent monoids used in all such examples are not finitely generated. We proceed to provide an example of a positive semiring of the form $E(M)$ for a finitely generated monoid $M$ such that $E(M)$ is neither a bi-HFS nor a bi-LFS.

\smallskip

\begin{example}
	Let $M$ be the numerical monoid $M = \langle 2, 3 \rangle = \nn_0 \setminus \{1\}$, and consider the positive semiring $S = E(M)$. Theorem~\ref{t:L-W} ensures that $(S,+)$ is a UFM. The elements $e^2$ and $e^3$ belong to $\mathcal{A}_\times(S)$ because they are atoms of the divisor-closed submonoid $e(M)$ of $(S^\bullet, \cdot)$. Therefore the equality $(e^2)^3 = (e^3)^2$ reveals that $(S^\bullet, \cdot)$ is not an HFM, and so $S$ is not a bi-HFS. Since the monoids $e(M)$ and~$M$ are isomorphic, it follows from \cite[Example~2.13]{CS11} that $e(M)$ is an LFM (see also \cite[Corollary~3.3]{CCGS21}). However, we will verify that $(S^\bullet, \cdot)$ is not an LFM. Consider the element $e^2 + e^3$ of~$S$ and write $e^2 + e^3 = \big( \sum_{i=0}^m a_i e^i \big) \big( \sum_{j=0}^n b_j e^j \big)$ taking $a_0, \dots, a_m$ and $b_0, \dots, b_n$ in~$\nn_0$ with $a_m b_n \neq 0$ and $a_1 = b_1 = 0$. As $(S,+)$ is free on $\{e^j \mid j \in M\}$ by Theorem~\ref{t:L-W}, it follows that $a_m b_n e^{m+n} = e^3$, whence $a_m = b_n = 1$ and $m+n = 3$. Assuming that $m < n$, we obtain that $m=0$ and $n=3$, from which $\sum_{i=0}^m a_i e^i = e^0 = 1$. Thus, $e^2 + e^3 \in \mathcal{A}_\times(S)$. In a similar way, we can verify that $e^2+ 2e^3 + e^4 \in \mathcal{A}_\times(S)$. Now the equality $(e^2+e^3)(e^2+e^3) = e^2(e^2 + 2e^3 + e^4)$ allows us to conclude that $(S^\bullet, \cdot)$ is not an LFM. Hence $S$ is not a bi-LFS.
\end{example}

\smallskip

It is clear that $\nn_0$ is a positive semiring that is a bi-UFS. As the reader may have already noticed, this is the only example of a bi-UFS that we have exhibited so far. Indeed, the aforementioned results lead us to make the following conjecture.

\smallskip

\begin{conj}\label{conj:bi-UFS}
	A positive semiring $S$ is a bi-UFS if and only if $S = \nn_0$. 
\end{conj}

\smallskip

In the direction of Conjecture~\ref{conj:bi-UFS}, we pose the following questions.

\smallskip

\begin{question} \label{quest:bi-HFS/LFS}
	\hfill
	\begin{enumerate}
		\item Is $\nn_0$ the only positive semiring that is a bi-HFS?
		\smallskip
		
		\item Is $\nn_0$ the only positive semiring that is a bi-LFS?
	\end{enumerate}
\end{question}

\smallskip

We now show that the positive semiring $E(\nn_0)$ is a bi-FFS that is neither a bi-HFS nor a bi-LFS. Lindemann-Weierstrass Theorem (Theorem~\ref{t:L-W}) guarantees that the semiring of polynomials $\nn_0[x]$ is isomorphic to the positive semiring $E(\nn_0)$ via $p(x) \mapsto p(e)$. As most readers should be more familiar with polynomial notation, in the next example we think of $E(\nn_0)$ in terms of polynomials.

\smallskip

\begin{example} \label{ex:bi-FFS neither bi-HFS nor bi-LFS}
	As $(\nn_0[x], +)$ is a free (commutative) monoid with basis $\{x^n \mid n \in \nn_0\}$, it is a UFM and hence an FFM. The multiplicative monoid $(\nn_0[x]^\bullet, \cdot)$ is also an FFM, as we proceed to argue. Fix $f(x)$ in $\nn_0[x]^\bullet$, and let $d(x) \in \nn_0[x]^\bullet$ be a divisor of $f(x)$ in $(\nn_0[x]^\bullet, \cdot)$. Then $d(x)$ divides $f(x)$ in the integral domain $\zz[x]$, which is clearly an FFD (that is, $(\zz[x]^\bullet, \cdot)$ is an FFM). It follows from \cite[Theorem~5.1]{AAZ90} that $f(x)$ has only finitely many non-associate divisors in $\zz[x]$, and so $f(x)$ has only finitely many divisors in $(\nn_0[x]^\bullet, \cdot)$. Hence $(\nn_0[x]^\bullet, \cdot)$ is an FFM by \cite[Theorem~2]{fHK92} and, as a result,~$S$ is a bi-FFS.
	
	We now illustrate that $(\nn_0[x]^\bullet, \cdot)$ is neither an HFM nor an LFM. By \cite[Corollary 2.2]{CCMS09}, for each $n \in \nn$, the polynomial $(x+n)^n(x^2 - x + 1)$ is irreducible in $(\rr_{\ge 0}[x]^\bullet, \cdot)$ and, therefore, in $(\nn_0[x]^\bullet, \cdot)$. As a result, for every $k \in \nn$, the expressions $[(x+n)^n(x^2-x+1)]\cdot[x+1]^k$ and $[x+n]^n[(x^2-x+1)(x+1)][x+1]^{k-1}$ are factorizations of the same element in $(\nn_0[x]^\bullet, \cdot)$ with lengths $k+1$ and $n+k$, respectively. Hence $(\nn_0[x]^\bullet, \cdot)$ is not an HFM. 
	On the other hand, it follows from \cite[Corollary 2.2]{CCMS09} that the polynomials $(x+1)(x^2-x+3)$ and $(x+2)(x^2-x+3)$ are irreducibles in $(\rr_{\ge 0}[x]^\bullet, \cdot)$ and so in $(\nn_0[x]^\bullet, \cdot)$. Since $[(x+1)(x^2-x+3)]\cdot[x+2]$ and $[(x+2)(x^2-x+3)]\cdot[x+1]$ are distinct factorizations of length $2$ of the same element in $(\nn_0[x]^\bullet, \cdot)$, we see that $(\nn_0[x]^\bullet, \cdot)$ is not an LFM. Thus, $\nn_0[x]$ is neither a bi-HFS nor a bi-LFS.
\end{example}

\smallskip

We now provide a concrete example of a positive semiring that is a bi-BFS but neither its additive monoid nor its multiplicative monoid are HFMs/LFMs.

\smallskip

\begin{example} \label{ex:bi-BFS that is not bi-HFS/LFS}
	Consider the positive semiring $S_2 = \{0,1\} \cup \mathbb R_{\geq 2}$. It is a bi-BFS by Theorem~\ref{thm:BF sufficient condition}. Since $\mathcal A_+(S_2) = \{1\} \cup (2,3)$, the equalities $5 \cdot 1 = 2 \cdot (5/2) = 7/3 + 8/3$ ensure that $(S_2,+)$ is neither an HFM nor an LFM. On the other hand, it follows from Theorem~\ref{thm:BF sufficient condition} that $[2,3]$ is contained in $\mathcal{A}_\times(S_2)$ and, in particular, $2,3,8/3, 14/5$, and $15/7$ belong to $\mathcal{A}_\times(S_2)$. As a result, the equality $2^3 = 3 \cdot (8/3)$ implies that $(S_2^\bullet, \cdot)$ is not an HFM and the equality $2 \cdot 3 = (14/5) \cdot (15/7)$ implies that $(S_2^\bullet, \cdot)$ is not an LFM.
\end{example}

\smallskip

It is clear that every bi-UFS is a bi-FFS, a bi-HFS, and a bi-LFS. This observation, along with Example~\ref{ex:bi-FFS neither bi-HFS nor bi-LFS}, allows us to conclude with an extended version of Diagram~\eqref{diag:AAZ's chain for semirings} for positive semirings. This extended diagram illustrates that, as it is the case for monoids and integral domains, most of the implications in Diagram~\eqref{diag:AAZ's chain for semirings} are not reversible in the context of positive semirings. Whether or not the topmost horizontal implications in Diagram~\eqref{diag:update AAZ's chain for semirings} are reversible is the gist of Conjecture~\ref{conj:bi-UFS} and Question~\ref{quest:bi-HFS/LFS}.

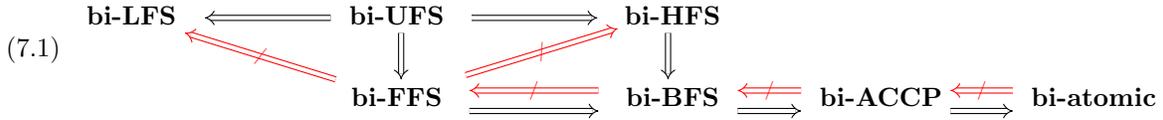
\smallskip

\begin{equation} \label{diag:update AAZ's chain for semirings}
	\begin{tikzcd}[cramped]
		\textbf{ bi-LFS \ } \arrow[rr, Leftarrow] && \textbf{ bi-UFS } \ \arrow[rr, Rightarrow]  \arrow[d, Rightarrow] && \ 
		\textbf{ bi-HFS } \arrow[d, Rightarrow] \\
		&& \textbf{  bi-FFS } \  \arrow[rr, Rightarrow,  shift left=-1ex]  \arrow[red, ull, Rightarrow, "/"{anchor=center,sloped}, shift left=0ex]  \arrow[red, urr, Rightarrow, "/"{anchor=center,sloped}, shift left=0.4ex] \arrow[red, rr, Leftarrow, "/"{anchor=center,sloped}, shift left=0.8ex] && \ 
		\textbf{ bi-BFS } \arrow[r, Rightarrow,  shift left=-1ex] \arrow[red, r, Leftarrow, "/"{anchor=center,sloped}, shift left=0.8ex] & 
		\textbf{ bi-ACCP}  \arrow[r, Rightarrow,  shift left=-1ex]  \arrow[red, r, Leftarrow, "/"{anchor=center,sloped}, shift left=0.8ex] & \textbf{ bi-atomic} %\\ 
%		\textbf{bi-LFS}
	\end{tikzcd}
\end{equation}
\medskip

\bigskip
%%%%%%%%%%%%%%%
%%%%%%%%%%%%%%%
\section*{Acknowledgments}

During the preparation of this paper, the third author was generously supported by the NSF award DMS-1903069.

\bigskip
%%%%%%%%%%%%%%
%%%%%%%%%%%%%%

\end{document}